\documentclass[12pt]{article}

\usepackage{mathrsfs}
\usepackage{bm}
\usepackage[T1]{fontenc}
\usepackage[latin1]{inputenc}
\usepackage{enumerate}
\usepackage{hyperref}
\usepackage{comment}
\usepackage{bbm}
\usepackage{stmaryrd}
\usepackage{latexsym}
\usepackage{amsfonts}
\usepackage{parskip}
\usepackage{bbm}
\usepackage{amsmath,amssymb,amscd}
\usepackage{yfonts}
\usepackage{dsfont}
\usepackage[dvips]{graphicx}
\usepackage{epsfig}
\usepackage{psfrag}
\usepackage[font={small}]{caption}
\usepackage{color}
\usepackage{float}
\usepackage{upgreek}
\usepackage{tikz}
\usepackage{tikz-cd}
\usepackage{relsize}
\usepackage{isomath}

\DeclareFontFamily{U}{txsyc}{}
\DeclareFontShape{U}{txsyc}{m}{n}{
   <-> txsyc%
}{}
\DeclareFontShape{U}{txsyc}{bx}{n}{
   <-> txbsyc%
}{}
\DeclareFontShape{U}{txsyc}{l}{n}{<->ssub * txsyc/m/n}{}
\DeclareFontShape{U}{txsyc}{b}{n}{<->ssub * txsyc/bx/n}{}
\DeclareSymbolFont{symbolsC}{U}{txsyc}{m}{n}
\SetSymbolFont{symbolsC}{bold}{U}{txsyc}{bx}{n}
\DeclareFontSubstitution{U}{txsyc}{m}{n}
\DeclareMathSymbol{\df}{\mathrel}{symbolsC}{"42}
\DeclareMathSymbol{\fd}{\mathrel}{symbolsC}{"43}
\DeclareMathSymbol{\lJoin}{\mathrel}{symbolsC}{"58}
\DeclareMathSymbol{\rJoin}{\mathrel}{symbolsC}{"59}

\newcommand{\cA}{{\cal A}}

\newcommand{\cE}{{\cal E}}

\newcommand{\cP}{{\cal P}}

\newcommand{\cV}{{\cal V}}

\newcommand{\EE}{\mathbb{E}}

\newcommand{\NN}{\mathbb{N}}

\newcommand{\RR}{\mathbb{R}}

\newcommand{\TT}{\mathbb{T}}
\newcommand{\ZZ}{\mathbb{Z}}

\newcommand{\frd}{\mathfrak{d}}

\newcommand{\fL}{\mathfrak{L}}

\newcommand{\fs}{\mathfrak{s}}

\newcommand{\di}{\displaystyle}
\newcommand{\iy}{\infty}
\newcommand{\lt}{\left}
\newcommand{\me}{\medskip}

\newcommand{\pa}{\partial}
\newcommand{\ri}{\rightarrow}
\newcommand{\rt}{\right}
\newcommand{\sm}{\smallskip}

\newcommand{\wi}{\widetilde}
\newcommand{\wit}{\widehat}

\newcommand{\Ent}{\mathrm{Ent}}

\DeclareMathOperator*{\esssup}{ess\,sup}
\newcommand{\ex}{\exists\ }
\newcommand{\fo}{\forall\ }

\newcommand{\Id}{\mathrm{Id}}

\newcommand{\lve}{\lt\vert}
\newcommand{\lVe}{\lt\Vert}

\newcommand{\rve}{\rt\vert}
\newcommand{\rVe}{\rt\Vert}

\newcommand{\st}{\,:\,}

\newcommand{\trace}{\mathrm{tr}}

\newcommand{\un}{\mathds{1}}
\newcommand{\Var}{\mathrm{Var}}

\newcommand{\vvv}{\vert\!\vert\!\vert}
\newcommand{\bq}{\begin{eqnarray*}}
\newcommand{\bqn}[1]{\begin{eqnarray}\label{#1}}
\newcommand{\eq}{\end{eqnarray*}}
\newcommand{\eqn}{\end{eqnarray}}
\newcommand{\wwtbp}{\par\hfill $\blacksquare$\par\me\noindent}
\newcommand{\thistitlepagestyle}{}
\newcommand{\lin}{\llbracket}
\newcommand{\rin}{\rrbracket}

\newcommand{\ttsim}{\raise.17ex\hbox{$\scriptstyle\mathtt{\sim}$}}

\newtheorem{pro}{Proposition}

\newtheorem{cor}[pro]{Corollary}
\newtheorem{lem}[pro]{Lemma}
\newtheorem{theo}[pro]{Theorem}

\newtheorem{remark}[pro]{Remark}

\renewcommand{\thepro}{\arabic{pro}}
\newenvironment{con}
{\par\me\refstepcounter{pro}\noindent{\bf Conjecture \thepro\ }}
{\par\hfill $\square$\par\sm\noindent}
\newenvironment{deff}
{\par\me\refstepcounter{pro}\noindent{\bf Definition \thepro\ }}
{\par\hfill $\square$\par\sm\noindent}

\newenvironment{rem}{\par\me\refstepcounter{pro}\noindent{\bf Remark \thepro\ }}
{\par\sm\noindent}

\newcommand{\proof}{\par\me\noindent\textbf{Proof: }} 
\newcommand{\prooff}[1]{\par\me\noindent\textbf{#1}\par\sm\noindent}

\newcommand{\esD}{\mathbf{D}}
\newcommand{\esL}{\mathbf{L}}
\newcommand{\esB}{\mathbf{B}}

\newcommand{\esF}{\mathbf{F}}
\newcommand{\esC}{\mathbf{C}}

\newcommand{\esFf}{\esF_{\hskip-.3mm\mathrm{f}}}

\renewcommand{\liminf}{\mathop{\underline{\lim}}\limits}
\renewcommand{\limsup}{\mathop{\overline{\lim}}\limits}

\newcommand{\lb}{\left (}
\newcommand{\rb}{\right )}

 \renewcommand{\L}{\mathbf{L}}

    \def\C{{\mathbb C}}
    \def\P{{\mathbb P}}

    \def\E{{\mathbb E}}
    \def\N{{\mathbb N}}

    \def\i{\textnormal {i}}
    \def\d{{\textnormal d}}
    \def\R{{\mathbb R}}

\newcommand{\discret}[1]{\mathds{#1}}
\newcommand{\cont}[1]{{#1}}

\newcommand{\A}{\mathbf{A}}

\newcommand{\Gate}[3]{#1\stackrel{#2}{\curvearrowright}#3}

\newcommand{\inn}[2]{\langle #1, #2 \rangle}

\newcommand{\Cmir}[4]{#1\stackrel{#2}{\curvearrowright}#3\stackrel{#4}{\curvearrowright}#1}
\newcommand{\wu}{\tau}
\newcommand{\bwu}{\boldsymbol{\wu}}

\newcommand{\Pwu}{P_\wu}
\newcommand{\Pswu}{P^{\bwu}}
\newcommand{\wPwu}{\wi{P}_\wu}
\newcommand{\wPswu}{\wi{P}^{\bwu}}

\newcommand{\cmir}{interweaving relation}
\newcommand{\cmirs}{interweaving relations}

\newcommand{\Mwu}{\P(\wu \in dt)}

\newcommand{\LaLa}{\Lambda_{{\beta_{\varepsilon}}}}
\newcommand{\LaLad}{\Lambda^*_{\beta_{\varepsilon}}}

\newcommand{\gdbeta}{{\beta+\varepsilon}}
\newcommand{\ptbeta}{\varepsilon}
\newcommand{\deltabeta}{\beta}

\title{On  interweaving relations}
\author{Laurent Miclo${}^\dagger$ and Pierre Patie${}^\ddagger$}
 \date{\vbox{\copy0
\vskip5mm
\copy1
}
 }

\begin{document}

\setbox0=\vbox{
\large
\begin{center}
Toulouse School of Economics, UMR 5314\\
Institut de Mathématiques de Toulouse, UMR 5219,\\
CNRS and University of Toulouse
\end{center}
}

 \setbox1=\vbox{
 \large
 \begin{center}
 ${}^\ddagger$ School of Operations Research and Information Engineering,\\
 Cornell University \\
\end{center}
}
\setbox5=\vbox{
\hbox{miclo@math.cnrs.fr\\}
\vskip1mm
\hbox{Toulouse School of Economics, \\}
\hbox{Manufacture des Tabacs, 21, Allée de Brienne\\}
\hbox{31015 Toulouse cedex 6, France\\}
\vskip1mm
\hbox{Institut de Mathématiques de Toulouse\\}
\hbox{Université Paul Sabatier, 118, route de Narbonne\\}
\hbox{31062 Toulouse cedex 9, France\\}
\hbox{miclo@math.univ-toulouse.fr\\}
}
\setbox6=\vbox{
\hbox{${}^\ddagger$ pp396@cornell.edu\\}
\vskip1mm
\hbox{School of Operations Research and Information Engineering\\}
\hbox{Cornell University\\}
\hbox{Ithaca, NY 14853\\}
\hbox{USA\\}
 }

\maketitle
\thistitlepagestyle
\abstract{
Interweaving relations are introduced and studied here in a general Markovian setting as a strengthening of usual
intertwining relations between  semigroups, obtained by adding a randomized delay feature. They provide a new classification scheme of the set of Markovian semigroups which  enables to transfer from a reference semigroup and up to an independent warm-up time, some ergodic, analytical and mixing properties  including the $\varphi$-entropy convergence to equilibrium, the hyperboundedness and when the warm-up time is deterministic the cut-off phenomena. We also present several useful transformations that preserve  interweaving relations.  
We provide a variety of examples of  interweaving relations ranging from classical, discrete, and non-local Laguerre and Jacobi semigroups to degenerate hypoelliptic Ornstein-Uhlenbeck semigroups and some non-colliding particle systems.
}

{\small
}\par

\tableofcontents


\section{Introduction and main results}

Comparison and classification  are  traditional mathematical tools to transfer information from a reference object to more complex ones.
The goal of this paper is to develop this  framework in the study of Markov semigroups by  introducing the notion of  interweaving  as a refinement of the usual concept of  intertwining. Anticipating the formal definition given below, an interweaving relation between two Markov semigroups can be seen as a symmetric (or a two-sided) intertwining relations between them with the additional feature that the two Markovian intertwining kernels factorize one of the semigroup considered at  a  random time.

The recent years have witnessed the ubiquity and usefulness of intertwining relations in the study of Markov processes. Indeed, this concept which traces back
to the works of Dynkin \cite{Dynkin} and Rogers and Pitman \cite{RP} yielding, in that later case,  at the relationship between
a Brownian motion in $\mathbb{R}^n$ and its radial part, the Bessel process of dimension $n$, has  been, for instance, used by Diaconis and Fill \cite{MR1071805} in relation  with strong stationary  times, by Carmona, Petit and Yor \cite{MR1654531} in relation to the so-called self-similar  saw tooth-processes, extended by Patie and  Savov in \cite{Patie-Savov-GeL, Patie-Savov-BG}  to general self-similar positive Markov processes, by Miclo \cite{miclo:hal-01281029} in connection with the algebraic concept of similarity transform, by Fill \cite{Fi} for an elegant characterization of
the distribution of the first passage time of some Markov chains, by Borodin and Olshanski \cite{BO,BO2} for the construction of Markov processes on infinite dimensional spaces, by 
S.~Pal and M.~Shkolnikov \cite{PalS} for diffusions,  by Patie and Simon \cite{Patie-Simon} and Patie and Zhao \cite{Patie-Zhao} in relation with fractional operators.

The concept of interweaving will reinforce this line of research by proposing further developments in the investigation of general Markov processes.  Although additional applications can certainly be developed, we will primarily focused on the study of ergodic, analytical and mixing properties of Markov semigroups including, for instance,   convergence to equilibrium in the sense of $\varphi$-entropy, hyperboundness properties and cut-off phenomena. Our  range of examples will be very broad as it encompasses some discrete Markov chains, classical linear diffusions, some denegenerate hypoelliptic diffusions, stochastic dynamics on partitions and some Markov processes with jumps.

\par\me


Let us now proceed  with the formal definition of interweaving relations between Markov semigroups. Consider a (measurable) Markov kernel semigroup $P\df(P_t)_{t\geq 0}$  on a measurable state space $(V,\cV)$.
Namely, $P$ is a Markov kernel from $\RR_+\times V$ to $V$:
for any $A\in \cV$, the function $\RR_+\times V\ni (t,x)\mapsto P_t(x,A)$ is measurable
and for any $(t,x)\in \RR_+\times V$, the mapping $\cV\ni A\mapsto P_t(x,A)$ is a probability measure.
The semigroup property asserts that for any $t,s\geq 0, \, P_{t}P_s=P_{t+s}$, in the sense of the composition of Markov kernels  from $V$ to $V$. 
Let now $\wi P\df(\wi P_t)_{t\geq 0}$ be another Markov semigroup on a measurable state space $(\wi V,\wi\cV)$.
We say there is a (Markov) \textbf{intertwining relation} from $P$ to $\wi P$ when there exists a Markov kernel $\Lambda$ from $V$ to $\wi V$ such that
\bqn{inter}
\fo t\geq 0,\qquad P_t\Lambda&=&\Lambda \wi P_t.\eqn
It will be convenient to denote $\Gate{P}{\Lambda}{\wi P}$ this
commutation property (or $\Gate{P_t}{\Lambda}{\wi P_t}$ for the relation between Markov kernels for a fixed $t\geq 0$).
Such a link may not say much. For instance when $\wi P$ admits an invariant probability $\wi\nu$, \eqref{inter} is satisfied by considering the Markov kernel $\Lambda=\wi\nu$
defined by
\bqn{wimu}
\fo x\in V,\,\fo \wi A\in\wi\cV,\qquad\Lambda(x,\wi A)&=& \wi \nu(A).\eqn
\par
The intertwining relation \eqref{inter} is said to be \textbf{symmetric}  when  there exists another Markov kernel $\wi\Lambda$ from $\wi V$ to $V$ such that
$\Gate{\wi P}{\wi \Lambda}{P}$.
A more meaningful notion is the following one.
\begin{deff}
  We say that $\cont{P}$ has an  \textbf{interweaving relation}  with  $\wi{P}$ if there exist two Markov kernels ${\Lambda}$ and $\wi{\Lambda}$ and a non-negative  random variable $\wu$ such that
  \begin{eqnarray}\label{eq:cmir}
  && \Cmir{\cont{P}}{{\Lambda}}{\wi{P}}{\wi{\Lambda}}    \\
  \Lambda \wi{\Lambda} &  =&\Pwu=\int_{0}^{\infty} P_t \: \Mwu. \label{warm-up}
  \end{eqnarray}
We call $\wu$ the \textbf{warm-up time} or the \textbf{delay} and we  write  $\cont{P} \looparrowleft \wi{P}$  or $\cont{P} \stackrel{\wu}{\looparrowleft} \wi{P}$  to emphasize the dependency on $\wu$. Note that when $\wu=\delta_{{t}_0}$ is the degenerate random variable at  ${t}_0>0$, we may simply write, when there is no confusion,  $\cont{P} \stackrel{{t_0}}{\looparrowleft} \wi{P}$.

\noindent When $\wu$ is in addition infinitely divisible we say that $\cont{P}$ admits an \textbf{interweaving relation with an infinitely divisible} warm-up time (for short \textbf{IRID}) with  $\wi{P}$ and  we  write   $\cont{P} \stackrel{\boldsymbol{\wu}}{\looparrowleft} \wi{P}$.

\noindent Finally, when
we also have
\begin{equation} \label{eq:sym}
\wi\Lambda\Lambda=\wPwu
\end{equation}
we say that  there is a \textbf{symmetric  interweaving relation} between $P$ and $\wi P$ and we write $\cont{P} \stackrel{\wu}{\leftrightsquigarrow} \wi{P}$ (resp.~$\cont{P} \stackrel{\boldsymbol{\wu}}{\leftrightsquigarrow} \wi{P}$ when $\wu$ is infinitely divisible).
  \end{deff}
 Note that due to our measurability assumption above on the kernel $P$, the integrand in the r.h.s.\ of \eqref{eq:cmir} is measurable with respect to $t>0$ and the identity can be understood as the  Markov kernel on $(V,\cV)$ defined by
\bq
\fo x\in V,\,\fo A\in\cV,\qquad \Pwu(x,A)&\df& \int_0^{+\iy} P_t(x,A)\, \Mwu \eq
\par
 The notion of interweaving is related to completely monotone functions. Indeed,   observe that
  \begin{equation} \label{eq:subP}
   \Pwu=\int_{0}^{\infty} e^{-tL} \Mwu=F(L)
   \end{equation}
   where $L$ is the infinitesimal  generator of $P$ and $F$ as the Laplace transform of positive measure is, by Bochner classical result,  a completely monotone function, i.e.~$F\in C^{\infty}(\R_+)$ and $(-1)^{n}F^{(n)}(x)\geq 0$ for all $n\in \N$ and $x\geq0$. Next,
we recall that a random variable $\wu$ is said to be infinitely divisible if for each $N\in \N$, there exits a sequence of i.i.d.~random variables $(\wu_{n})_{1\leq n \leq N}$ such that, in distribution, $\wu \stackrel{(d)}{=}\wu_1+\ldots \wu_n$.
  Note that when $\wu$ is in addition infinitely divisible then there exists a Bernstein function $\phi$, i.e.~$\phi(0)\geq 0$ and $\phi'$ is completely monotone, such that, in \eqref{eq:subP} above, $F=e^{-\phi}$. Moreover, in such a case, there exists an unique convolution semigroups on $\R^+$ whose transition kernel is the law of a subordinator  $\boldsymbol{\wu}=(\wu_t)_{t\geq 0}$, a non-decreasing  L\'evy process, such that $\wu\stackrel{(d)}{=}\wu_1$ and
   $\Pswu=(\Pswu_t)_{t\geq 0}$ is a Markov semigroup, where for any bounded Borelian function $f$ and $t\geq 0$,
 \begin{equation}\label{def:Ptau}\Pswu_tf = \int_{0}^{\infty} P_sf \: \P(\wu_t \in ds). \end{equation}
 $\Pswu$ is the subordination  of $P$ in the sense of Bochner   and we have $\Pswu_1=\Pwu$.
 The  definition of interweaving can be summarized by
the following commutative diagram (suggesting the name of interweaving), holding for every $t\geq 0$:
\begin{figure}[H]\centering
\begin{tikzcd}[scale=2]
V\arrow[r, "P_t"] \arrow[d, "\Lambda"'] \arrow[dd, bend right=50 , "\Pwu"']
& V\arrow[d, "\Lambda" ]  \arrow[dd, bend left=50 , "\Pwu"]\\
\wi V\arrow[r, "\wi P_t" ]\arrow[d, "\wi\Lambda"']
&\wi V\arrow[d, "\wi\Lambda" ]\\
V\arrow[r, "P_t"]
& V
\end{tikzcd}
\caption{Interweaving relations with $\Lambda\wi\Lambda=\Pwu$}\label{fig1}
\end{figure}
\par
Our objective in this paper is to provide some properties and investigate some applications  of interweaving relations  in the study of probabilistic and analytical properties of general Markov processes. 
Before presenting its range of applications, let us present a few general observations about this concept.

\subsubsection*{Some general comments on interweaving relations}

\begin{enumerate}[(a)]\label{1errems}
\item  The above Markov framework is quite plain. There are several ways to enrich it, especially to associate a generator $L$ to the semigroup $P$, since this is in general
the simplest way to describe $P$. 
Analytically, the semigroup $P$ can be acting on a Banach space, in the sense of Hille-Yosida theory, see e.g.\ the book of Yosida \cite{MR96a:46001}.
One standard choice, when $P$ admits an invariant probability $\nu$, is to consider the Hilbert space $\esL^2(\nu)$.
Another possibility, when the state space $V$ is endowed with a $\sigma$-compact topology, is to consider the space of continuous functions vanishing at infinity, endowed with the supremum norm.\par
From a probabilistic point of view, the generator $L$ appears in the formulation of an underlying martingale problem
for the trajectories $X\df (X_t)_{t\geq 0}$ of an associated Markov process
(cf.\ for instance the book of Ethier and Kurtz \cite{MR838085}). Usually the state space $V$ is endowed with a topology and the trajectories are càdlàg,
in particular the position $X_t$ converges to $X_0$ as $t$ goes to $0_+$.
\par
The examples considered in this paper will be described  through their generators. All will admit an invariant measure which will be a probability measure, except for the squared Bessel processes and some related examples,
 and thus the $\esL^2$ setting and the martingale problems will be equivalent.\par
 As $t$ goes to zero and in the appropriate senses dictated by the above analytical or probabilist frameworks, $P_t$ converges to the identity operator $\Id$, seen as the transition kernel corresponding to no motion.
\item  When  the generators $L$ and $\wi L$ are available for the semigroups $P$ and $\wi P$, e.g.\ in one of the meanings seen in (a),
the intertwining relation \eqref{inter} is often equivalent to $L\Lambda=\Lambda \wi L$, where the Markov kernel $\Lambda$
has to be seen as an operator from $\esD(\wi L)$ to (a subset of) $\esD(L)$, the respective domains of the generators.
When the intertwining relation is symmetric, see \eqref{eq:sym}, we should have that the image of $\esD(L)$ by $\wi \Lambda$ is  included in $\esD(\wi L)$,
in particular for interweaving relations, the l.h.s.\ of \eqref{warm-up} can also be seen as an operator from $\esD(L)$ to itself which can be ``extended'' into $\Pwu$, a priori acting on
$\esB(V)$, the space of bounded measurable functions on $V$. 
\item  One way to avoid the case \eqref{wimu} is to ask for $\Lambda$ to be one-to-one, e.g.\ as an operator from $\esB(\wi V)$  to $\esB(V)$ (but when $\wi V$ is not discrete, this is often requiring too much). Somewhat the requirement \eqref{warm-up} also goes in this direction: in the ``regular'' situations described above in (a), $P_t$ converges to $\Id$ for small $t>0$
and thus should end up being invertible in this asymptotic. This should still be true for $\Pwu$ when $\wu$ has a distribution concentrated near $0$ and in particular $\Lambda$ would be one-to-one and $\wi \Lambda$ would be surjective.
In the case of a symmetric interweaving relation  with
a warm-up variable $\wu$
on $\RR_+$ concentrated near $0$, we can expect $\Lambda$ and $\wi\Lambda$ to be both invertible.
That is why, more generally and heuristically, we see symmetric
interweaving    as a Markovian formulation of a weak invertibility assumption on $\Lambda$ and $\wi\Lambda$, resulting in $P$ and $\wi P$ being closely related.
In the same spirit, the  more mass the law of $\wu$ gives to neighborhoods of $0_+$, the more informative \eqref{warm-up} is, as the ``invertibility of $\Pwu$ should be stronger''.
Conversely, assuming that $P$ is ergodic with invariant probability measure $\nu$, we have that for large $t\geq 0$, $P_t$ is converging to $\nu$ (seen as a Markov kernel as in \eqref{wimu}).
It follows that the more the law of $\wu$ is concentrated on large values, the less informative \eqref{warm-up} becomes.
This interpretation will be strengthened when we will see $\wu$ as  a random warm-up time. 
\item  From a
spectral point of view and in the regular settings of (a), the meaning of an interweaving relation   from $P$ to $\wi P$
seems to be that the spectrum of the generator $L$ of $P$ is included into the spectrum of the generator $\wi L$ of $\wi P$, at least under appropriate ergodicity assumptions on $P$ and when the spectrum is be understood in an extended sense.
We will not enter into the underlying technicalities here, so let us just mention a conjecture that we hope to investigate in a future work:
\begin{con}\label{con1}
Consider  two irreducible Markov generators $L$ and $\wi L$ on finite
 state spaces $V$ and $\wi V$.
There exists a interweaving relation   from  $(\exp(tL))_{t\geq 0}$ to $(\exp(tL))_{t\geq 0}$
if and only if the extended spectrum of $L$ is included into that of $\wi L$. By extended spectrum, we mean the eigenvalues as well as the dimensions of the associated Jordan blocks (inclusion implying smaller or equal dimensions).
\end{con}
Such a result and possible
 extensions to more general state spaces would provide a spectral understanding of why interweaving relations  enable to deduce quantitative informations on the
convergence to equilibrium for $P$ from similar knowledge from $\wi P$. 
\item  Assume an intertwining relation $\Gate{P}{\Lambda}{\wi P}$ and that $P$ and $\wi P$ admit  reversible probability measures $\mu$ and $\wi\mu$, with
$\mu\Lambda=\wi\mu$.
Working in the $\esL^2$ framework mentioned above in (a), we get by duality an intertwining relation $\Gate{\wi P}{\Lambda^*}{P}$.
A priori $\Lambda^*\st \esL^2(\mu)\ri \esL^2(\wi\mu)$ is an abstract Markov operator, in the sense that it preserves non-negativity and the function always taking the value 1.
To get a ($\esL^2$-)interweaving relation, it remains to check that $\Lambda\Lambda^*=\Pwu$. Thus in such a reversible setting, interweaving relations  are relatively easy to deduce from intertwining relations.
\item  Assume that we have a symmetric intertwining relation between two semigroups $P$ and $\wi P$, namely
$\Gate{P}{\Lambda}{\wi P}$ and $\Gate{\wi P}{\wi\Lambda}{P}$ for some Markov kernels $\Lambda$ and $\wi\Lambda$.
Then necessary $\Lambda\wi\Lambda$ commutes with all the $P_t$ for $t\geq 0$.
Assume that $P$ admits a generator $L$ which is diagonalizable with eigenvalues of multiplicities one.
When functional calculus is available, we deduce that $\Lambda\Lambda^*$ is of the form
$F(-L)$, where $F\st \RR_+\ri \RR$ is a measurable mapping. To get a interweaving relation is then equivalent to $F$ being completely monotone.
\item  The symmetric interweaving relation  does not correspond to the symmetrization of the interweaving relation,
which is only requiring two interweaving relations, one from $P$ to $\wi P$ and one from $\wi P$ to $P$.
For the latter,
 the kernels from $\wi V$ to $ V$ and from $V$ to $\wi V$
 may be different from $\wi\Lambda$ and $\Lambda$,
as well as the warm-up time from $\wu$.
Some results below can be extended from symmetric to symmetrized interweaving relations
But the notion of symmetric interweaving relation  is natural because of Proposition \ref{symmetric} below.
\end{enumerate}
\par\me

\subsection{Basic properties of interweaving relations} \label{sec:basic_prop}
We present  now some useful transformations of semigroups that  preserve  interweaving relations and postpone their proofs to Section \ref{poTaP}.
We start with the following  result that enables to construct from an IRID  with a random warm-up time a interweaving relation with the constant $1$ as  warm-up time. This observation will be useful in some applications of interweaving relations for which the assumption of  deterministic warm-up time is required.
\begin{theo}\label{theo3b}
Assume that $\cont{P} \stackrel{\bwu}{\looparrowleft} \wi{P}$, that is the warm-up time $\wu$ is infinitely divisible. Then $\Pswu \stackrel{1}{\looparrowleft} \wPswu   $ where  $\bwu=(\wu_t)_{t\geq 0}$ is the subordinator 
such that $\wu\stackrel{(d)}{=}\wu_1$ and the subordinated semigroups are defined as in \eqref{def:Ptau}.
\end{theo}
\par\me
We point out that  in Section \ref{dwte} (resp.~Section \ref{sec:rwt}),  we present several examples for which the  warm-up time $\wu$ is a constant (resp.~a positive infinitely divisible random variable). In the applications of interweaving relations to ergodic properties, the previous result allows us to compare the approach based on interweaving relations with the classical ones based on functional inequalities.

We proceed with additional properties of interweaving relations. To simplify the forthcoming discussion,    we assume that $P$ (resp.~$\wi{P}$) is a semigroup on some Banach space $\mathbf B$ (resp.~$ \mathbf{\wi B}$), e.g.~if $P$ is a Feller semigroup then $\mathbf B= \textrm{C}_b(V)$ the space of continuous and bounded functions on $V$ endowed with the supremum topology.  

Let us now come to symmetric interweaving relations. They are a consequence of interweaving relations   under a seemingly mild additional assumption:
\begin{pro}\label{symmetric}
When the Markov kernel $\Lambda$ is one-to-one,
say from $\esB(\wi V)$ to $\esB(V)$,
then a interweaving relation  is symmetric.
\end{pro}
\par
\proof
Indeed, from \eqref{warm-up}, we deduce, first for a  non-negative Borelian function $f$ and then for a general Borelian function $f$, by writing $f=\max(f,0) - \max(-f,0)$, that
\bq\Lambda\wi\Lambda\Lambda f &=&\Pwu\Lambda f=\int_0^{+\iy}P_t\Lambda f \, \Mwu=\int_0^{+\iy}\Lambda \wi P_t\, \Mwu f
=\Lambda\wPwu  f\eq
where we used Tonnelli theorem for the last identity. The injectivity of $\Lambda$ implies that $\wi\Lambda\Lambda=\wPwu$.\wwtbp
The  one-to-one assumption of Proposition \ref{symmetric} is quite restrictive, when the state spaces are not denumerable.
Nevertheless, the simplicity of the above proof shows
it can be weakened when working in the Hille-Yosida framework mentioned in Remark \ref{1errems}(a), by considering the corresponding notion of
injectivity, in particular in $\esL^2$ spaces.
\par\me

We now proceed by showing  that,  under mild conditions, $\looparrowleft$  is an equivalence relation. This highlights the idea, triggered by this concept,  of an original classification scheme which enables to extend in a natural way to general Markov semigroups some ergodic and analytical properties  that were attainable only for some specific classes, such as reversible  diffusion ones.

\begin{theo}\label{thm:ref-cmir}
  Assume that the Markov intertwining kernels is one-to-one on  a dense subset of $\mathbf B$   then $\looparrowleft$ is an equivalence relation as
  \begin{enumerate}[(i)]
  \item \label{it:T0} $\looparrowleft$ is reflexive, that is  $\cont{P} \stackrel{0}{\looparrowleft} \cont{P}$  with $0$ the degenerate variable  at $0$. 
  \item \label{it:T1} $\looparrowleft$ is symmetric, that is if $\cont{P} \stackrel{\tau}{\looparrowleft} \wi{P}   $ then $\wi{P}   \stackrel{\tau}{\looparrowleft} \cont{P} $ with $\Cmir{\wi{P}  }{\wi{\Lambda}}{\cont{P}}{{\Lambda}}$ and $ \wi\Lambda\Lambda=\wPwu$.
  \item  \label{it:T2} $\looparrowleft$ is transitive, that is if  $\cont{P} \stackrel{\tau}{\looparrowleft} \wi{P}   $ and  $\wi{P}   \stackrel{\wi{\tau}}{\looparrowleft} \overline{P} $ then $\cont{P} \stackrel{\tau+\wi{\tau}}{\looparrowleft} \overline{P} $, where $\tau$ and $\wi{\tau}$ are assumed to be independent.
   \end{enumerate}
  \end{theo}

  \begin{remark}
    It is not difficult to check that if one restricts  the previous theorem to the  subset of IRID then  $\looparrowleft$ remains an equivalence relation.
   \end{remark}
   Schematically, the transitivity property of interweaving relations can be described, for any $t\geq 0$, where by rotating to $45$ degrees the figure of our previous diagrams:
 \begin{figure}[H]\centering
\begin{tikzcd}[scale=2]
V\arrow[d, "P_t"'] \arrow[r, "\Lambda"']
& \wi V \arrow[d, "\wi P_t"'] \arrow[r, "\Lambda' "'] \arrow[rr, bend left=50 , "\wi P_{\wi \tau}"']&v_\vee \arrow[ d, "\overline{P}_t"']   \arrow[r, "\wi \Lambda' "'] &\wi V \arrow[ d, "\wi P_t"'] \arrow[r, "\wi \Lambda "'] & V \arrow[ d, "P_t"']\\
V\arrow[r, "\Lambda"]
& \wi V\arrow[r, "\Lambda' "] \arrow[rr, bend right=50 , "\wi P_{\wi \tau}"]&v_\vee  \arrow[r, "\wi \Lambda' "] & \wi V \arrow[r, "\wi \Lambda "] & V
\end{tikzcd}
\caption{Transitive interweaving relations}\label{fig6}
\end{figure}
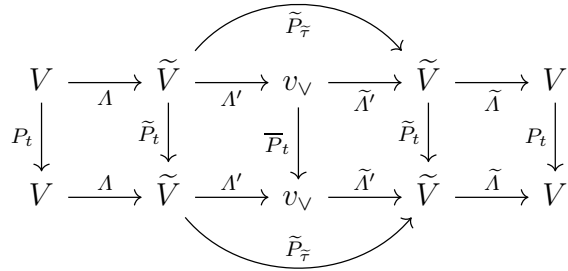

We proceed with the following theorem that provides  a closure  property of interweaving relations by similarity transform as well as a way  to \emph{transport} interweaving relations. 
\begin{theo}\label{thm:simil}\label{thm:transp}
Let us assume that  $\cont{P} \stackrel{\tau}{\looparrowleft} \wi{P} $.
\begin{enumerate}[1)]
\item  Let $P^M$ be a Markov semigroup acting on the Banach space $\mathbf B^M$ If the two Markov  $P$ and $P^M$ are similar, that is,  for all $t\geq 0$, $P_t^M = M P_t M^{-1}$ where $M$ and its inverse $M^{-1}$ are bounded operators. Then,
  \[ \cont{P^M} \stackrel{\tau}{\looparrowleft} \wi{P}   \] where,   with the obvious notation $\Lambda^M = M \Lambda$  and $\wi \Lambda^M = \wi \Lambda  M^{-1}$.
  \item  If \begin{equation}\label{eq:intpq}
  \Gate{\cont{P}}{T}{\discret{P}}  \textrm{ and }  \Gate{\wi{P}  }{T}{\wi{\discret{P}}}
  \end{equation}
  with $\discret{P}$ and $\wi{\discret{P}}$ two Markov semigroups defined on  the measurable space $(\discret{V},\mathscr{V})$  and $T$ an one-to-one Markov operator.
   Then \[\discret{P} \stackrel{\tau}{\looparrowleft} \wi{\discret{P}} \qquad \textrm{ and } \qquad  \discret{P}_\tau= {\boldsymbol{\Lambda}}\wi{\boldsymbol{\Lambda}} \] where
  \begin{equation}\label{eq:iLi}
  \Gate{\Lambda}{T}{\boldsymbol{\Lambda}} \textrm{ and }\Gate{\Lambda}{T}{\wi{\boldsymbol{\Lambda}}}.
   \end{equation}
\end{enumerate}
\end{theo}




\par

\subsection{Applications of interweaving relations to the theory of Markov semigroups} \label{sec:appl}
We now turn  to the description
of some interesting features and applications of interweaving relations.
Thorughout this section, we make the hypothesis that $P$ and $\wi P$  admit $\nu$ and $\wi\nu$ as invariant probability measures, respectively,
and  $\cont{P} \stackrel{\tau}{\looparrowleft} \wi{P}   $ . In this case, $\nu\Lambda$ is also an invariant probability measure for $\wi P$,
as shown by multiplying  \eqref{inter} on the left by $\nu$. Similarly, $\wi\nu\wi\Lambda$ is invariant for $P$.
We will assume that $\wi\nu=\nu\Lambda$ and that
$\nu=\wi\nu\wi\Lambda$, when the invariant probability measures are not unique.\par

\subsubsection{Entropy convergence to equilibrium}
We want to deduce estimates on the speed of convergence of $P$ to the equilibrium $\nu$ by taking into account a similar knowledge for $\wi P$ and $\wi \nu$.
First we must specify the way to measure how far a probability measure $m$ on $V$ is from  $\nu$ and here we choose the  entropy (see Subsection \ref{ecte} for extension of the result to general $\varphi$-entropy).
The \textbf{(relative) entropy} of  $m$ with respect to $\nu$  is given by
\bq
\Ent(m\vert\nu)&\df&\lt\{
\begin{array}{ll}\di \int \ln\lt(\frac{dm}{d\nu}\rt)\, dm, &\hbox{if $m\ll\nu$}\\
+\iy,&\hbox{otherwise}\end{array}\rt.
\eq
where $dm/d\nu$ stands for the Radon-Nikodym density of $m$ with respect to $\nu$.
As desired, the quantity $\Ent(m\vert\nu)$ measures the discrepancy between $m$ and $\nu$, in particular we have the Pinsker's bound:
\bq
\Ent(m\vert\nu)
&\geq &2 \lVe m-\nu\rVe_{\mathrm{tv}}^2
\eq
where the \textbf{total variation} distance $ \lVe m-\nu\rVe_{\mathrm{tv}}$ between $m$ and $\nu$ is defined as the supremum  of  $m(A)-\nu(A)$ over $A\in\cV$.
\par
We proceed by assuming  that we have some information about the convergence of $\wi P$ towards $\wi\nu$, under the following form:
there exists a function $\varepsilon\st \RR_+\times \overline{\RR}_+\ri \overline{\RR}_+$, with $\overline{\RR}_+\df\RR_+\sqcup\{+\iy\}$, which is
non-decreasing with respect to the second variable,
such that
\bqn{vareps}
\fo \wi m_0\in\cP(\wi V),\,\fo t\geq 0,\qquad \Ent(\wi m_0\wi P_t\vert\wi\nu)&\leq &\varepsilon(t, \Ent(\wi m_0\vert\wi\nu))\eqn
where $\cP(\wi V)$ is the set of all probability measures on $\wi V$.
For this bound to be meaningful, we furthermore require that
\bq
\fo E\in\RR_+,\qquad \lim_{t\ri +\iy} \varepsilon(t,E)&=&0\eq
A typical instance of \eqref{vareps} is when $\wi P$ satisfies (modified) logarithmic Sobolev inequalities (here and below,  we refer for instance to the book of Ané et al.\ \cite{MR2002g:46132} for  a friendly presentation of these inequalities).
Then there exists a constant $\wi\alpha>0$ such that
\eqref{vareps} holds with the function $\varepsilon$ given by
\bq
\fo t\geq 0,\,\fo E\in\overline{\RR}_+,\qquad
\varepsilon(t,E)&=&\exp(-\wi\alpha t)E\eq
\par\sm

Here is the transfer of the entropic convergence estimate to $P$:
\begin{theo}\label{theo1}
Assume that $\cont{P} \stackrel{\tau}{\looparrowleft} \wi{P}   $  and that
\eqref{vareps} holds. Then we have
\bqn{espE}
\fo  m_0\in\cP(V),\,\fo t\geq 0,\qquad \Ent( m_0P_{t+\wu}\vert\nu)&\leq &\varepsilon(t, \Ent(m_0\vert\nu))\eqn
\end{theo}
\par
From a probabilistic point of view (see Remark \ref{1errems}(a) or the definition of a measurable Markov process below),  $m_0P_{t+\wu}$ is the distribution of
$X_{t+\wu}$, where $\tau$ is a random time independent of $X$ and distributed according to $\wu$.
The bound \eqref{espE} says that up to waiting a random warm-up time $\wu$, we get for $P$ the same estimate on the speed of convergence to equilibrium
as for $\wi P$.

\subsubsection{Hypercontractivity}
Another famous classical application of logarithmic Sobolev inequalities concerns hypercontractivity, which is a kind of regularizing property.
Interweaving relations   equally enable its transfer from a semigroup to another one, up to a random warm-up time.
More precisely, the \textbf{hypercontractivity} property of the semigroup $\wi P$, is the existence of a constant $\wi\alpha>0$ (which may be different from the one considered above,
for Markov processes which are not diffusions), such that
we have for the operator norms,
\bqn{hypercon}
\fo t\geq 0,\qquad
\vvv \wi P_t\vvv_{\esL^2(\wi\nu)\ri\esL^{p(\wi\alpha t)}(\wi\nu)}&\leq &1\eqn
where
\bq
\fo t\geq 0,\qquad p(\wi\alpha t)&\df&1+\exp(\wi\alpha t) \eq
\par
Here is the analogue of Theorem \ref{theo1} for hypercontractivity:
\begin{theo}\label{theo2}
Assume that $\cont{P} \stackrel{\wu}{\looparrowleft} \wi{P}   $   and that
\eqref{hypercon} holds. Then we have
\bqn{eqtheo2}
\fo t\geq 0,\qquad  \vvv P_{t+\wu}\vvv_{\esL^2(\nu)\ri\esL^{p(\wi\alpha t)}(\nu)}&\leq & 1\eqn
\end{theo}
\par\me

\subsubsection{Cut-off phenomenon} \label{sec:cutoff}
Coming back to the convergence to equilibrium, we now explain  how a
symmetric interweaving relation   enables the transfer of the cut-off phenomenon (for a short survey of this notion, see Diaconis \cite{MR1374011}).
To state our result, we need a family $(P^{(n)})_{n\in\ZZ_+}$ of Markov semigroups on state spaces $(V^{(n)})_{n\in\ZZ_+}$ with respective invariant probability measures $(\nu^{(n)})_{n\in\ZZ_+}$. Defining, for any $n\in\ZZ_+$,
\bqn{frdn}
\fo t\in\RR_+,\qquad  \frd^{(n)}(t)&\df& \sup_{m_0\in \cP(V^{(n)})}\lVe m_0P^{(n)}_{  t}-\nu^{(n)} \rVe_{\mathrm{tv}}\eqn
we say  that the family $(P^{(n)})_{n\in\ZZ_+}$ has
\begin{enumerate}[(1)]
  \item \label{def:cut1}a (uniform) \textbf{cut-off}  at the positive \textbf{cut-off times}  $(t^{(n)})_{n\in\ZZ_+}$ when
  for any $r\in (0,1)\sqcup(1,+\iy)$,
\bq\lim_{n\ri\iy} \frd^{(n)}(rt^{(n)})&=&\un_{\{0<r<1\}}\eq
  \item a  \textbf{window cut-off} (resp.~\textbf{profile cut-off})  at   $(t^{(n)},w^{(n)})_{n\in\ZZ_+}$ (resp.~and with profile $\eta$) if  $t^{(n)} \rightarrow \infty$, $w^{(n)}=o(t^{(n)})$ as $n\to \infty$, and
\bq \lim_{c\to -\infty}\liminf_{n\ri\iy} \frd^{(n)}(t^{(n)}+c w^{(n)})=1 \textrm{ and } \lim_{c\to +\infty} \limsup_{n\ri\iy}\frd^{(n)}(t^{(n)}+c w^{(n)})=0
\eq
(resp.~and  for all $c\in \R$, ${\eta}(c)=\liminf_{n\ri\iy} \frd^{(n)}(t^{(n)}+c w^{(n)})=\limsup_{n\ri\iy} \frd^{(n)}(t^{(n)}+c w^{(n)})$). 

\end{enumerate}

With these definitions we have the following result.
\begin{theo}\label{theo3}
Consider two families of  Markov semigroups $(P^{(n)})_{n\in\ZZ_+}$ and $(\wi P^{(n)})_{n\in\ZZ_+}$  on  $(V^{(n)})_{n\in\ZZ_+}$ and $(\wi V^{(n)})_{n\in\ZZ_+}$ and with invariant probability distributions
$(\nu^{(n)})_{n\in\ZZ_+}$ and $(\wi\nu^{(n)})_{n\in\ZZ_+}$, respectively.
Let  $(t^{(n)})_{n\in\ZZ_+}$ be a sequence of  positive real numbers and assume that for any $n\in\ZZ_+$, $\cont{P} \stackrel{\boldsymbol{t}^{(n)}_0}{\leftrightsquigarrow} \wi{P} $
 such that
\bqn{neglig}
\lim_{n\ri\iy} \frac{\boldsymbol{t}^{(n)}_0}{t^{(n)}}&=&0 \quad (\textrm{resp.~}
\limsup_{n\ri\iy} \frac{\boldsymbol{t}^{(n)}_0}{w^{(n)}}=0)
\eqn
Then the cut-off  (resp.~\textbf{window cut-off} and \textbf{profile cut-off}) phenomenon  with cut-off times $(t^{(n)})_{n\in\ZZ_+}$ (resp.~windows $(w^{(n)})_{n\in\ZZ_+}$ and  profile $\eta$) for $(P^{(n)})_{n\in\ZZ_+}$
is equivalent to that of $(\wi P^{(n)})_{n\in\ZZ_+}$.
\end{theo}
\par
The remaining part of the paper is organized  as follows. In the two forthcoming  sections we describe several  examples of interweaving relations along with their applications.  More specifically, in the next section, we focus on interweaving relations  where the warm-up distribution is a Dirac mass: this includes  the two points space and the intertwining relations between continuous and discrete Bessel and Laguerre processes and some degenerate hypoelliptic Ornstein-Uhlenbeck processes.
In Section 3, we consider  interweaving relations   between diffusive Laguerre processes of different parameters, as well as some semigroups associated to Markov processes with jumps. Finally we prove extensions of the statements presented in this introduction in Section~4.

\section{Deterministic warm-up time examples}\label{dwte}

Three examples of  \cmirs\   whose warm-up times are deterministic are presented in the following subsections: there exists $t_0\geq 0$ such that
$\tau=\delta_{t_0}$. In this situation the statements of Theorems \ref{theo1} and  \ref{theo2} simplify, as \eqref{espE} and \eqref{eqtheo2}
are respectively replaced by
\bq
\fo  m_0\in\cP(V),\,\fo t\geq 0,\qquad \Ent( m_0P_{t_0+t}\vert\mu)&\leq &\varepsilon(t, \Ent(m_0\vert\mu))\eq
and
\bq
\fo t\geq 0,\qquad  \vvv P^{(\beta)}_{t_0+t}\vvv_{\esL^2(\mu)\ri\esL^{p(\wi\alpha t)}(\mu)}&\leq & 1\eq
\par

\subsection{The two point space}\label{2points}

Consider the simplest non-trivial case of the setting of the introduction, where $V=\wi V$ is the  two point space $\{0,1\}$.
Let $L$ and $\wi L$ be two isospectral irreducible Markov generators on $V$.
We can write
\bq
L&=&\lambda(\mu-\Id)\eq
 where $\lambda>0$ is the non-zero eigenvalue of $-L$, $\mu$ is the invariant probability of $L$, seen as a Markov kernel, and $\Id$
is the identity operator. Any non-zero function $\varphi$ on $V$ such that $\mu[\varphi]=0$ is an eigenfunction of $L$ associated to the eigenvalue $-\lambda$.
Consider the function $\varphi$ normalized in $\esL^2(\mu)$ given by
\bqn{varphi}
\varphi&\df& \lt(\begin{array}{c}\varphi(0)\\
\varphi(1)\end{array}\rt) \ \df\  \lt(\begin{array}{c} l\\
-1/l\end{array}\rt)\qquad \hbox{with}\ l\ \df\ \sqrt{\frac{\mu(1)}{\mu(0)}}\eqn
Since $\wi L$ is irreducible and isospectral with $L$, it can be written $\lambda(\wi\mu-\Id)$, where $\wi\mu$ is the invariant probability of $\wi L$. Define  $\wi\varphi$ as in \eqref{varphi}, with $\mu$  replaced by $\wi\mu$.\par
 For $\epsilon>0$,
define $ \Lambda_{\epsilon}$ the linear mapping sending $\wi\varphi$ to $\epsilon\varphi$ and preserving the function $\un$.
It is immediate to check that $\Gate{L}{\Lambda_{\epsilon}}{\wi L}$.
A priori $\Lambda_{\epsilon}$ is not a Markov kernel. Nevertheless its matrix in the basis $(\un_{\{0\}}, \un_{\{1\}})$ is of the form
$\lt(\begin{array}{cc} a&1-a\\
b&1-b\end{array}\rt)$ and we compute that
\bq
a&=&\frac{1+\epsilon l\wi l}{1+\wi l^2}\\
b&=&\frac{1-\epsilon \wi l/l}{1+\wi l^2}\eq
It follows that for $\epsilon>0$, $\Lambda_{\epsilon}$ is Markovian if and only if
\bqn{largesteps}
\epsilon&\leq & \min(l/\wi l\, ,\, \wi l/l)\eqn
 Choose  $\epsilon_0\df  \min(l/\wi l\, ,\, \wi l/l)$, the largest value such that $\Lambda_{\epsilon_0}$ is Markovian.
Symmetrically, for $\wi\epsilon>0$, construct $\wi\Lambda_{\wi\epsilon}$ sending $\varphi$ to $\wi\epsilon\wi\varphi$ and preserving $\un$.
We have $\Gate{\wi L}{\wi\Lambda_{\wi\epsilon}}{ L}$ and by symmetry of the r.h.s.\ of \eqref{largesteps}, $\wi \Lambda_{\epsilon}$ is Markovian for $\wi\epsilon\in(0,\epsilon_0]$.
Again choose $\wi\epsilon=\epsilon_0$, the largest value such that $\wi\Lambda_{\wi\epsilon}$ is Markovian.
The mapping $\wi\Lambda_{\epsilon_0}\Lambda_{\epsilon_0}$
 is uniquely determined by the fact that it preserves $\un$ and that $(\wi\Lambda_{\epsilon_0}\Lambda_{\epsilon_0})\varphi=\epsilon_0^2\varphi$.
This observation leads us to consider $t_0\df t_0(L,\wi L)\df -\ln(\epsilon_0^2)\geq 0$, so that $\wi\Lambda_{\epsilon_0}\Lambda_{\epsilon_0}=\exp(t_0 L)$. We are thus in the framework considered in the introduction.
Similarly, we  get $\Lambda_{\epsilon_0}\wi\Lambda_{\epsilon_0}=\exp(t_0\wi L)$, and this can also be deduced from Proposition \ref{symmetric}, since $\Lambda_{\epsilon_0}$ is invertible. It seems that $t_0$ is
the smallest warm-up deterministic time enabling to go from estimates of convergence
for one of the semigroup to the other one. As in the introduction, let us consider more specifically the traditional case of relative entropy.
Diaconis and Saloff-Coste \cite{MR1410112} computed the logarithmic Sobolev constant $\alpha(L)$ of $L$:
\bq
\alpha(L)&=&4\frac{1-2\mu_\wedge}{\ln(1/\mu_\wedge-1)}\lambda\eq
with $\mu_\wedge\df \mu(0)\wedge\mu(1)$, the smallest value taken by the invariant measure.
\par
We have for any initial distribution $m_0$ on $\{0,1\}$,
\bqn{DiacSal}
\fo t\geq 0,\qquad \Ent(m_0\exp(t L)\vert\mu)&\leq & \exp(- \alpha(L)t)\Ent(m_0\vert\mu)\eqn
\par
Taking into account Theorem \ref{theo1}, this bound can be improved into
\bq
\fo t\geq 0,\qquad \Ent(m_0\exp(t L)\vert\mu)&\leq & \min\{ \exp(-\alpha(\wi L)(t-t_0(L,\wi L))_+)\st \wi L\in\fL(L)\}\Ent(m_0\vert\mu)\eq
where $\fL(L)$ is the set of irreducible Markov generators isospectral to $L$.
Note that $\alpha(\wi L)$  is   strictly decreasing as a function of $\wi\mu_\wedge$ and thus
 the logarithmic Sobolev constants of $L$ and $\wi L$ are distinct when $L\not=\wi L$ (up to the symmetry exchanging 0 and 1).
Furthermore, the bound $\alpha(\wi L)\leq 2\lambda$ is only attained  when $\wi\mu$ is the uniform distribution on $\{0,1\}$ (in this case the computation of the logarithmic Sobolev inequality is due
to Gross \cite{MR0420249}). So it is appealing to try a comparison with this ``fastest case'' where $\wi\mu= (1/2, 1/2)$, and we get
\bqn{improved2pt}
\fo t\geq 0,\qquad \Ent(m_0\exp(t L)\vert\mu)&\leq & \exp(- 2\lambda (t-\ln(1/\mu_\wedge-1))_+)\Ent(m_0\vert\mu)\eqn
since
\bq
\epsilon_0\ =\ \min\lt( \sqrt{\frac{\mu(1)}{\mu(0)}},  \sqrt{\frac{\mu(0)}{\mu(1)}}\rt)\ =\ \sqrt{\frac{\mu_\wedge}{1-\mu_\wedge}}\eq
so that $t_0=\ln(1/\mu_\wedge-1)$.\par
Formula \eqref{improved2pt} becomes rapidly better than \eqref{DiacSal}.
It follows that, for ``medium'' times,  to get good estimates of the relative entropy with respect to $\mu$ of the time marginal laws of the Markov evolution generated by $L$, it is more interesting
to intertwine this evolution with the isospectral generator $\wi L$ corresponding to  the uniform distribution than to compute the logarithmic Sobolev constant associated to $L$.
\par
The existence of Markovian kernels $\Lambda$ and $\wi \Lambda$ intertwining two irreducible  isospectral (in the extended sense: equality of eigenvalues and dimensions of the Jordan blocks) and finite Markov generators was shown in \cite{miclo:hal-01281029}.
We believe these kernels can furthermore be chosen so that a  \cmir\   holds, as a subcase of Conjecture \ref{con1}.

\subsection{Classical and discrete squared Bessel processes}

The examples  described in this subsection and in the following one were the first instances of  \cmirs\   that we identified  in
 \cite{self1}. However, this notion was not properly isolated and investigated there. 

 \noindent For a given $\beta>0$, consider the classical squared Bessel diffusion generator $G_\beta$ of index $\beta-1$ (dimension $2\beta$) on $\RR_+$ given by
\bq
 \fo x\in(0,+\iy),\qquad G_{\beta}&\df&x\pa^2+\beta\pa\eq
 where $\pa$ is the usual differentiation operator.
  This diffusion generator admits $\mu_\beta$ as invariant (even reversible) measure, where
\bq
\fo x\in(0,+\iy),\qquad \mu_{\beta}(dx)&\df &\frac{x^{\beta-1}}{\Gamma(\beta)}\,dx\eq
where $\Gamma$ is the usual gamma function. For $\beta>0$, denote $Q^{(\beta)}\df(Q^{(\beta)}_t)_{t\geq 0}$ the Markov semigroup generated by $G_\beta$.
 \par\sm
An analogue discrete  squared Bessel birth-and-death generator $ \discret{G}_{\beta}$ is defined by
\bq  \fo n\in \ZZ_+, \qquad\discret{G}_{\beta}&\df&(n+\beta) \pa_+ + n\pa_-\eq
where the operators $\pa_{\pm}$ act on any function $\mathbfit{f}\st \ZZ_+\ri\RR$ via
\bq\fo n\in\ZZ_+,\qquad
\pa_{\pm}\mathbfit{f}(n)\df& \mathbfit{f}(n\pm 1)-\mathbfit{f}(n)\eq
(with the convention that $\mathbfit{f}(-1)\df \mathbfit{f}(0)$).
 The birth-and-death generator $\discret{G}_{\beta}$ admits ${\mathbfit{u}}_\beta$ as invariant (even reversible) measure, where
\bq
\fo n\in\ZZ_+,\qquad {\mathbfit{u}}_{\beta}(n)&\df &\frac{(n+\beta-1)(n+\beta-2)\cdots \beta}{n!}.\eq
For $\beta,\,\sigma>0$, denote $\discret{Q}^{(\beta,\sigma)}\df(\discret{Q}^{(\beta,\sigma)}_t)_{t\geq 0}$ the Markov semigroup generated by $\sigma\discret{G}_{\beta}$.
For $\sigma>0$, consider $\Lambda_\sigma$ the  Markov kernel from $\RR_+$ to $\ZZ_+$ given by the  Poisson transition probability measures:
 \bq
 \fo x\in\RR_+,\,\fo n\in\ZZ_+,\qquad \Lambda_\sigma(x,n)&\df& \frac{(\sigma x)^n}{n!}\exp(-\sigma x)\eq
 \par
 Conversely, for $\beta,\sigma>0$, consider $\wi \Lambda_{\beta,\sigma}$ the  Markov kernel from $\ZZ_+$ to $\RR_+$ given by the  gamma transition probability measures:
\bq
\fo n\in\ZZ_+,\,\fo x\in(0,\iy ),\qquad \wi{\mathbfit{\Lambda}}_{\beta,\sigma}(n,dx)&=& \sigma^{n+\beta}\frac{x^{n+\beta-1}}{\Gamma(n+\beta)}\exp(-\sigma x)\, dx\eq
\par
In  \cite{self1}, we have shown the following symmetric  \cmir\   with deterministic warm-up time $\sigma>0$:
\begin{pro}\label{pro3}
For any $\beta,\sigma>0$, we have
\begin{equation*}
   Q^{(\beta)} \stackrel{\sigma}{ \leftrightsquigarrow } \discret{Q}^{(\beta,\sigma)}
\end{equation*}
where $\Lambda=\Lambda_\sigma$ and $ \wi \Lambda = \wi{\mathbfit{\Lambda}}_{\beta,\sigma}$.
\end{pro}
\par
For $\beta>0$, the invariant measures $\mu_\beta$ and ${\mathbfit{u}}_\beta$ have infinite weight so the above Bessel processes do not enter in the framework of convergence to equilibrium
and we cannot apply the results presented in the introduction. Nevertheless the  \cmirs\   of Proposition \ref{pro3}
 are useful for simulation purposes of one process
in terms of the other one, especially in the direction of using the birth-and-death process to simulate the diffusion process, as it was seen in \cite{self1}.
\par\sm
\subsubsection{Non-colliding discrete and continuous squared Bessel processes}
We proceed by describing a very elegant  extension of the interweaving relations between squared Bessel processes to the multidimensional setting that has been recently proposed by Assiotis \cite{Assiotis}. More specifically, for any integer $N\geq 1$  and $\beta >0$, let $Q^{N,{\beta}}$ (resp.~$\discret{Q}^{N,{\beta}}$)  be the semigroup of  $N$ independent copies of squared Bessel processes (resp. the discrete squared Bessel process)   of index $\beta-1$ conditioned to never intersect. These semigroups are known to be Feller semigroups acting  on the space $C_0(W_+^N)$ and $\discret{C}_0(\discret{W}_+^N)$ respectively where the Weyl chambers with positive coordinates are defined by
 \begin{align*}
W^N_{+}&=\{ {\bf{x}}=(x_1,\cdots,x_N)\in \mathbb{R}_+^N: x_1\le x_2  \le \cdots \le x_N \} \\
\discret{W}^N_{+}&=\{ \textbf{n}=(n_1,\cdots,n_N)\in \mathbb{Z}_+^N: n_1< n_2 <  \cdots < n_N \}.
\end{align*}
  Then relying on the one-dimensional result that appeared in  \cite[Proposition 13 and 14]{self1},  Assiotis  obtain the following, see \cite[Proposition 1, Theorem 1.4, Remark 1.6]{Assiotis}.
  \begin{pro}
For any integer $N\geq 1$  and $\beta >0$, we have
\begin{equation*}
   Q^{N,(\beta)} \stackrel{1}{ \leftrightsquigarrow} \discret{Q}^{N,(\beta)}
\end{equation*}
where  $\Lambda=\Lambda^N_1$ and  $ \wi \Lambda = \wi{\mathbfit{\Lambda}}^N_{\beta,1}$ are Markov kernels defined respectively, for any $\textbf{n} \in \discret W^N_{+}$ and $\textbf{x} \in W^N_{+}$, by
\begin{align*}
\Lambda_1^N\left(\textbf{x},\textbf{n}\right)&=\frac{\Delta_N(\textbf{n})}{\Delta_N(\textbf{x})}\det \left(\Lambda_1(x_i,n_j)\right)_{i,j=1}^N, \\
\wi{\mathbfit{\Lambda}}^N_{\beta,1}\left(\textbf{n},d\textbf{x}\right)&=\frac{\Delta_N(\textbf{x})}{\Delta_N(\textbf{n})}\det \left(\wi{\mathbfit{\Lambda}}_{\beta,1}(n_i,dx_j\right)_{i,j=1}^Ndx_1\cdots dx_N,
\end{align*}
and $\Delta_N(\textbf{x})=\det\left(x_i^{j-1}\right)^N_{i,j=1}=\prod_{1 \le i <j \le N}^{}(x_j-x_i)$ stands for the Vandermonde determinant.
  \end{pro}

We mention that  the Markov realizations of the semigroups  $Q^{N,(\beta)}$ and $\discret Q^{N,(\beta)}$ appear in random matrix theory as the dynamics of the eigenvalues of the so-called continuous and discrete Laguerre ensembles and refer to \cite{Assiotis} for further connections between these objects and other algebraic structures.

\subsection{Classical and discrete Laguerre processes}\label{cadLp}

A natural way to transform the transient Bessel processes into recurrent processes is recalled in \cite{self1} and it leads to the Laguerre processes.
This procedure slightly modifies  the  \cmirs\   and we ended up with the following results.

\noindent For $\beta,\sigma>0$, consider the classical \textbf{Laguerre differential operator} $L_{\beta,\sigma}$  on $\RR_+$ acting on $\esC^{\iy}_{\mathrm{b}}(\RR_+)$, the space of  bounded smooth functions with bounded derivatives on $\RR_+$, via
\bqn{Lsigma}
\fo f\in\esC^{\iy}_{\mathrm{b}}(\RR_+),\,\fo x\in(0,+\iy),\qquad
L_{\beta,\sigma}[f](x)&=& \sigma x\pa^2f(x)+(\sigma\beta -x)\pa f(x) \eqn
This operator is a one-dimensional diffusion generator and it is easy to check that its unique invariant  (even reversible) probability measure $\nu_{\beta,\sigma}$ on $\RR_+$,
 is the \textbf{gamma distribution} of shape parameter $\beta$ and scale parameter $\sigma$, i.e.
\bq
\fo x\in(0,+\iy ),\qquad \nu_{\beta,\sigma}(dx)&=& \frac{x^{\beta-1}\exp(-x/\sigma)}{\sigma^{\beta}\Gamma(\beta)}\, dx\eq
It follows (via Freidrichs theory, see e.g.\ the book of Akhiezer and Glazman \cite{MR615737}) that $L_{\beta,\sigma}$ can be extended into a self-adjoint operator on $\esL^2(\nu_{\beta,\sigma})$.
The associated continuous Markov semigroup is denoted $P^{(\beta,\sigma)}\df(P^{(\beta,\sigma)}_t)_{t\geq 0}$.
\par\sm \noindent
An analogue discrete Laguerre birth-and-death generator $ \discret{L}_{\beta,\sigma}$ is defined by
\bqn{dLg}  \fo n\in \ZZ_+, \qquad\discret{L}_{\beta,\sigma}&\df&\sigma(n+\beta) \pa_+ + (\sigma+1)n\pa_-\eqn
This generator admits an invariant (even reversible) probability measure $\wi{\mathbfit{v}}_{\beta,\sigma}$ on $\ZZ_+$,
which is the negative binomial distribution of parameters $\beta$ and $\sigma/(1+\sigma)$, i.e.
\bq
\fo n\in\ZZ_+,\qquad
\wi{\mathbfit{v}}_{\beta,\sigma}(n)&\df&(1+\sigma)^{-\beta}\lt(\frac{\sigma}{\sigma+1}\rt)^n\frac{(n+\beta-1)(n+\beta-1)\cdots \beta}{n!}
\eq
Denote $\discret{P}^{(\beta,\sigma)}\df(\discret{P}^{(\beta,\sigma)}_t)_{t\geq 0}$ the Markov semigroup generated by $\discret{L}_{\beta,\sigma}$. In  \cite{self1}, we have shown the following symmetric  \cmir\   with deterministic warm-up time.
\begin{pro}\label{pro4}
For any $\beta,\sigma,\varsigma>0$, we have
\begin{equation*}
   P^{(\beta,\varsigma)} \stackrel{\ln(1+\frac{1}{\varsigma\sigma})}{\leftrightsquigarrow} \discret{P}^{(\beta,\varsigma\sigma)} 
\end{equation*}
where $\Lambda =\Lambda_\sigma$ and $ \wi \Lambda = \wi\Lambda_{\beta,\sigma+\frac1\varsigma}$.
\end{pro}
The last relation can be seen as a consequence of the last-but-one identity, via Proposition \ref{symmetric}, since $\Lambda_\sigma$, from \cite[Lemma 2.2]{self1}, is  one-to-one.
The relations  of Proposition \ref{pro4} can be summarized by the following  diagram:
\par
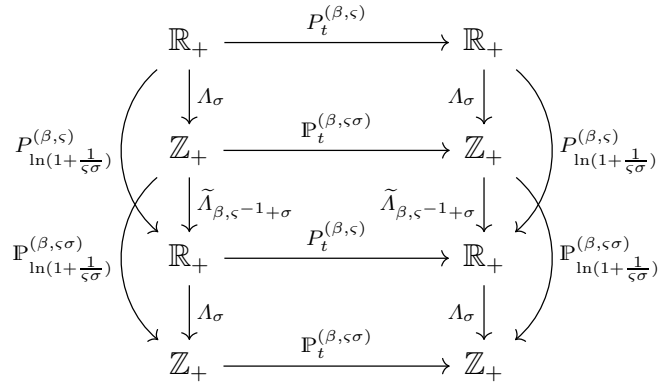
\begin{figure}[H]\centering
\begin{tikzcd}[column sep=30mm,scale=2]
{\RR_+ }\arrow[r,  "P^{(\beta,\varsigma)}_t"] \arrow[d,
"\Lambda_{\sigma}"] \arrow[dd, bend right=50 ,
"P_{\ln(1+\frac{1}{\varsigma\sigma})}^{(\beta,\varsigma)}"']
& \RR_+ \arrow[d, "\Lambda_{\sigma}"' ]  \arrow[dd, bend left=50 , "P_{\ln(1+\frac{1}{\varsigma\sigma})}^{(\beta,\varsigma)}"]\\
{\ZZ_+} \arrow[r,
"\discret{P}^{(\beta,\varsigma\sigma)}_{t}" ]\arrow[d, "\wi\Lambda_{\beta,\varsigma^{-1}+\sigma}"]\arrow[dd, bend right=50 , "\discret{P}_{\ln(1+\frac{1}{\varsigma\sigma})}^{(\beta,\varsigma\sigma)}"']
&
{\ZZ_+} \arrow[d,
"\wi\Lambda_{\beta,\varsigma^{-1}+\sigma}"' ]\arrow[dd, bend left=50 , "\discret{P}_{\ln(1+\frac{1}{\varsigma\sigma})}^{(\beta,\varsigma\sigma)}"]\\
{\RR_+} \arrow[r,
 "P^{(\beta,\varsigma)}_t"] \arrow[d, "\Lambda_{\sigma}" ]
&
{\RR_+} \arrow[d, "\Lambda_{\sigma}"' ]\\
\ZZ_+ \arrow[r, "\discret{P}^{(\beta,\varsigma\sigma)}_{t}"]
& \ZZ_+
\end{tikzcd}
\caption{Laguerre intertwining relations}\label{fig4}
\end{figure}\par

The  \cmirs\   between the continuous and discrete Laguerre processes enable to deduce links between their speed of convergence
to equilibrium. As in the introduction, let us present them in the
usual entropy sense (see Section \ref{poTaP} for generalisations).
First we recall the logarithmic Sobolev inequalities satisfied by the Laguerre semigroups.\par\sm
We start with the classical situation.
For any $\beta,\varsigma >0$, the \textbf{logarithmic Sobolev constant} $\alpha(\beta,\varsigma)$ associated to the generator $L_{\beta,\varsigma}$ defined in \eqref{Lsigma} is
\bqn{albv}
\alpha(\beta,\varsigma)&\df&\inf_{f\in\esC_{\mathrm{b}}^1(\RR_+)\st \nu_{\beta,\varsigma}[f^2]=1}\frac{4\varsigma\int_{\RR_+} xf^{\prime 2}(x)\, \nu_{\beta,\varsigma}(dx)}{\int_{\RR_+} f^2(x)\ln(f^2(x))\, \nu_{\beta,\varsigma}(dx)}
 \eqn
 (for any $k\in\NN$, $\esC_{\mathrm{b}}^k(\RR_+)$ is the space of bounded continuously $k$ times differentiable functions on $\RR_+$, with bounded derivatives).
 The numerator in \eqref{albv} is four times the \textbf{Dirichlet form (energy)} $\cE_{\beta,\varsigma}(f,f)$ associated to $L^{(\beta,\varsigma)}$ and defined, at least for $f\in \esC_{\mathrm{b}}^2(\RR_+)$,
 by
 \bq
 \cE_{\beta,\varsigma}(f,f)&\df&-\nu_{\beta,\varsigma}[f L_{\beta,\varsigma}[f]]\\
 &=&\varsigma\int_{\RR_+} x f^{\prime 2}(x)\, \nu_{\beta,\varsigma}(dx)\eq
 where the last equality is obtained by integration by parts and the last expression enables to extend the domain of definition of $\cE_{\beta,\varsigma}$.\par
 It is well-known (see for instance the book of Ané et al.\ \cite{MR2002g:46132}) that the logarithmic Sobolev constant is bounded above by twice the spectral gap of the associated generator.
 In the present setting, it implies that $\alpha(\beta,\varsigma)\leq 2$ for any $\beta,\varsigma>0$, since the spectrum of
 $L_{\beta,\varsigma}$ is $-\ZZ_+$ with eigenvalues of multiplicity 1, and so its spectral gap is 1.
In fact the constant $\alpha(\beta,\varsigma)$ does not depend on $\varsigma$:
 \begin{lem}\label{alpbet}
 For any $\beta,\varsigma >0$, we have $\alpha(\beta,\varsigma)=\alpha(\beta)$, where
 $\alpha(\beta)\df\alpha(\beta,1)$.
 \end{lem}
 \begin{rem} \label{rem:lsclag}
 The constant $\alpha(\beta)$ has been well-studied. Via the famous $\Gamma_2$-criterion, Bakry \cite{MR1417973} has shown that $\alpha(\beta)= 1$ for all $\beta\geq 1/2$.
Otherwise, the behavior of $\alpha(\beta)$ changes when $\beta>0$ is going to $0_+$, since it converges to zero  as
 $\alpha(\beta) \sim -4/\ln \beta $, see \cite{MR1971588}. We also refer to Corollary for an alternative analysis based on the concept of interweaving relation  of the convergence to equilibrium in entropy for $0\leq \beta <\frac12$.
 \end{rem}
 \proof
For any $\varsigma>0$, let $M_\varsigma$ be the dilation operator acting on
any function $f$ defined on $\RR_+$ via
\bq
 M_\varsigma f (x)&=&f(\varsigma x)\eq
An immediate linear change of variable shows that
 for any $\beta,\varsigma >0$, we have $\nu_{\beta,\varsigma}=\nu_\beta M_\varsigma$ (where $\nu_\beta$ stands for $\nu_{\beta,1}$).  For $f\in\esC_{\mathrm{b}}^1(\RR_+)$ with $\nu_{\beta,\varsigma}[f^2]=1$, consider the function $\wi f \df M_\varsigma f$.
 We have on the one hand,
 \bq\nu_{\beta,\varsigma}[f^2]&=&\nu_{\beta}[\wi f^2]\\
 \int_{\RR_+} f^2(x)\ln f^2(x) \, \nu_{\beta,\varsigma}(dx)&=& \int_{\RR_+} \wi f^2(x)\ln \wi f^2(x)\, \nu_{\beta}(dx)\eq
 and on the other hand,
 \bq \int_{\RR_+} xf^{\prime 2}(x)\, \nu_{\beta,\varsigma}(dx)&=&\frac{1}{\varsigma}\int_{\RR_+} x\wi f^{\prime 2}(x)\, \nu_{\beta}(dx)
 \eq
 The announced result now follows from the bijectivity of the mapping $f\mapsto \wi f$
between
 $\{f\in\esC_{\mathrm{b}}^1(\RR_+)\st \nu_{\beta,\varsigma}[f^2]=1\}$ and
 $\{\wi f\in\esC_{\mathrm{b}}^1(\RR_+)\st \nu_{\beta}[\wi f^2]=1\}$.
 \wwtbp
 \par
Here we are interested in $\alpha(\beta)$ since for any initial distribution $m_0$ on $\RR_+$,
 we have
 \bqn{conEntL}
 \fo t\geq 0,\qquad \Ent(m_0P_t^{(\beta,\varsigma)}\vert\nu_{\beta,\varsigma})&\leq & \exp(-\alpha(\beta)t)\Ent(m_0\vert\nu_{\beta,\varsigma})\eqn
 (of course, such a bound is only relevant when the initial relative entropy $\Ent(m_0\vert\nu_{\beta,\varsigma})$ is finite) and $\alpha(\beta)$ is optimal for these equalities to hold
 for any initial distribution $m_0\in\cP((0,+\iy))$ and for any time $t\geq 0$.
\par\sm
The quantitative convergence to equilibrium in the entropy sense has not been investigated for the  discrete Laguerre generators.
A priori, we have the following information.
 For $\beta,\sigma>0$, the modified logarithmic Sobolev constant ${\mathbfit{\alpha}}_m(\beta,\sigma)$ associated to the generator $ \discret{L}_{\beta,\sigma}$ defined in \eqref{dLg} is
\bqn{fabs}
{\mathbfit{\alpha}}_m(\beta,\sigma)&\df&\inf_{f\in\esFf(\ZZ_+)\st \mathbfit{v}_{\beta,\sigma}[\mathbfit{f}^2]=1}\frac{\mathds{E}_{\beta,\sigma}(\mathbfit{f}^2,\ln(\mathbfit{f}^2))}
{\mathbfit{v}_{\beta,\sigma}[ \mathbfit{f}^2\ln(\mathbfit{f}^2)]}
 \eqn
 where $\esFf(\ZZ_+)$ is the space of functions defined on $\ZZ_+$ which vanish except on a finite subset of points and where the Dirichlet form $\mathds{E}_{\beta,\sigma}(\mathbfit{f},\mathbfit{g})$ of two functions $\mathbfit{f},\mathbfit{g}\in \esFf(\ZZ_+)$ is given by
 \bq
 \mathds{E}_{\beta,\sigma}(\mathbfit{f},\mathbfit{g})&\df&-\mathbfit{v}_{\beta,\sigma}[\mathbfit{f}\mathds{L}_{\beta,\sigma}[\mathbfit{g}]]\\
 &=& \sum_{n\in\ZZ_+}(\mathbfit{f}(n+1)-\mathbfit{f}(n))(\mathbfit{g}(n+1)-\mathbfit{g}(n))\mathbfit{v}_{\beta,\sigma}(n)\mathds{L}_{\beta,\sigma}(n,n+1)
\eq
Again, the interest of ${\mathbfit{\alpha}}_m(\beta,\sigma)$ is the discrete analogue of \eqref{conEntL}: for any initial distribution $\mathbfit{m}_0$ on $\ZZ_+$,
we have
 \bqn{dconEntL}
 \fo t\geq 0,\qquad \Ent(\mathbfit{m}_0\discret{P}_t^{(\beta,\sigma)}\vert\mathbfit{v}_{\beta,\sigma})&\leq & \exp(-{\mathbfit{\alpha}}_m(\beta,\sigma)t)\Ent(\mathbfit{m}_0\vert\mathbfit{v}_{\beta,\sigma})\eqn
 (for the deduction of this bound and \eqref{conEntL} by differentiating their respective left-hand-side.\ with respect to the time $t\geq 0$, see again the book of Ané et al.\ \cite{MR2002g:46132}) and ${\mathbfit{\alpha}}_m(\beta,\sigma)$ is optimal for these inequalities  to hold
 for any initial distribution $\mathbfit{m}_0\in\cP(\ZZ_+)$ and for any time $t\geq 0$. We also have that ${\mathbfit{\alpha}}_m(\beta,\sigma)$ is bounded above by twice the spectral gap of $\mathds{L}_{\beta,\sigma}$.
Namely ${\mathbfit{\alpha}}_m(\beta,\sigma)\leq 2$ for any $\beta,\sigma>0$, since the spectrum of $\mathds{L}_{\beta,\sigma}$ is $-\ZZ_+$.
 Unfortunately, there is no proper way to estimate from below ${\mathbfit{\alpha}}_m(\beta,\sigma)$, which is only known in very few situations, especially related to the Poisson distribution, see Wu \cite{MR1800540}.
 That is why ${\mathbfit{\alpha}}_m(\beta,\sigma)$ is often replaced by the classical logarithmic Sobolev constant  $\mathbfit{\alpha}(\beta,\sigma)$, given by
 \bqn{wifabs}
{\mathbfit{\alpha}}(\beta,\sigma)&\df&\inf_{\mathbfit{f}\in\esFf(\ZZ_+)\st \mathbfit{v}_{\beta,\sigma}[\mathbfit{f}^2]=1}\frac{4\mathds{E}_{\beta,\sigma}(\mathbfit{f},\mathbfit{f}))}
{\mathbfit{v}_{\beta,\sigma}[ \mathbfit{f}^2\ln(\mathbfit{f}^2)]}
 \eqn
It can be checked that ${\mathbfit{\alpha}}(\beta,\sigma)\leq {\mathbfit{\alpha}}_m(\beta,\sigma)$, so that \eqref{dconEntL} still holds with ${\mathbfit{\alpha}}_m(\beta,\sigma)$ replaced by ${\mathbfit{\alpha}}(\beta,\sigma)$,
with the advantage that the latter ergodic constant can be estimated
  via
 discrete Hardy's inequalities
 (cf.\ \cite{MR1710983}):
 \par
 Consider the quantity
 \bq
 C_{\beta,\sigma}&\df& \min_{n\in\ZZ_+}\max(C_{\beta,\sigma}^-(n),C_{\beta,\sigma}^+(n))\eq
 where for any $n\in\ZZ_+$, we take
 \bq
 C_{\beta,\sigma}^-(n)&\df&\sup_{m<n}\lt(\sum_{l=m}^{n-1}\frac1{\mathbfit{v}_{\beta,\sigma}(l)\mathds{L}_{\beta,\sigma}(l,l+1)}\rt)\mathbfit{v}_{\beta,\sigma}(\lin 0, m\rin)\ln(1/\mathbfit{v}_{\beta,\sigma}(\lin 0, m\rin))\\
  C_{\beta,\sigma}^+(n)&\df&\sup_{m>n}\lt(\sum_{l=n}^{m-1}\frac1{\mathbfit{v}_{\beta,\sigma}(l)\mathds{L}_{\beta,\sigma}(l,l+1)}\rt)\mathbfit{v}_{\beta,\sigma}(\lin m, \iy\lin)\ln(1/\mathbfit{v}_{\beta,\sigma}(\lin m, \iy\lin))\eq\par
We have the general bounds
\bqn{Hardy}
\frac1{10}\frac1{C_{\beta,\sigma}}\ \leq \ {\mathbfit{\alpha}}(\beta,\sigma)\ \leq \ \frac83\lt(1-\frac{\sqrt{5}}{2\sqrt{2}}\rt)^{-1}\frac1{ C_{\beta,\sigma}}\eqn
These expressions can be exploited to get reasonably accurate estimates on ${\mathbfit{\alpha}}(\beta,\sigma)$ in terms of $\beta$ and $\sigma$, in particular ${\mathbfit{\alpha}}_m(\beta,\sigma)\geq {\mathbfit{\alpha}}(\beta,\sigma)>0$
for all $\beta,\sigma>0$ (insuring that the bound \eqref{dconEntL} is not trivial).
\par
Nevertheless, the underlying computations are not so nice, while resorting to  \cmirs\   eventually leads to better bounds on the convergence to equilibrium in the entropy sense.
More precisely, as a particular consequence of Theorem \ref{theo1} applied
to  the three last lines of Figure~\ref{fig4}, with $\varsigma=1$, we get
\begin{cor}\label{alterna}
For any initial probability $\mathbfit{m}_0$  on  $\ZZ_+$ and for any $\beta,\sigma>0$ and $t\geq 0$, we have
\bq
\Ent( \mathbfit{m}_0\discret{P}_t^{(\beta,\sigma)}\vert\mathbfit{v}_{\beta,\sigma})&\leq &\left(\frac{\sigma+1}{\sigma}\right)^{\alpha(\beta)} e^{-\alpha(\beta)t}  \: \Ent( \mathbfit{m}_0\vert\mathbfit{v}_{\beta,\sigma})
 \eq
 where we recall that $\alpha(\beta)=1$ for any $\beta\geq 1$, see Remark \ref{rem:lsclag}.
\end{cor}
\par
In particular, for $\beta \geq 1/2$ and
up to waiting a warming-up time $\ln(1+\frac{1}{\sigma})$, before which Corollary \eqref{alterna} provides no information and is less good than \eqref{dconEntL},
 we get after this period an exponential rate of convergence equal to $1$ (the best possible asymptotical one would be $2$, i.e.\ twice the spectral gap of $\mathds{L}_{\beta,\sigma}$).
Corollary \ref{alterna} is also relevant for small $\beta>0$, since one cannot hope for an estimate so  simple via \eqref{Hardy}.
\par\sm
Applying the bounds from Theorem \ref{theo1}
to  the three first lines of Figure~\ref{fig4}, we get for any initial probability $\mathbfit{m}_0$  on  $\ZZ_+$, any $\beta,\varsigma,\sigma>0$ and any $t\geq 0$,
\bq
\Ent(\mathbfit{m}_0\discret{P}_t^{(\beta,\varsigma)}\vert\mathbfit{v}_{\beta,\varsigma})&\leq & \exp(-{\mathbfit{\alpha}}_m(\beta,\varsigma\sigma)[t-\ln(1+\frac{1}{\varsigma\sigma})]_+)\Ent(\mathbfit{m}_0\vert\mathbfit{v}_{\beta,\varsigma})
 \eq
Letting $\sigma>0$ go to infinity and recalling that $\alpha(\beta)$ is optimal in \eqref{conEntL}, we deduce that
\bqn{suggest}
\fo \beta>0,\qquad \overline\alpha(\beta)\ \df\ \limsup_{\sigma\ri+\iy}{\mathbfit{\alpha}}_m(\beta,\sigma)&\leq & \alpha(\beta)\eqn
In particular
$  \overline\alpha(\beta)$ is going to zero as $\beta$ goes to $0_+$ (in fact we believe that $ \overline\alpha(\beta)=\alpha(\beta)$, as suggested by the remark about approximations at the end of this subsection).\par\me
Similar relations between the classical and discrete Laguerre semigroups are equally valid concerning hyperboundedness via Theorem \ref{theo2}.
Indeed, the logarithmic Sobolev inequalities imply that for any $\beta,\varsigma>0$, we have
\bq
\fo t\geq 0,\qquad
\vvv  P_t^{(\beta,\varsigma)}\vvv_{\esL^2(\nu_{\beta,\varsigma})\ri\esL^{p(\alpha(\beta) t)}(\nu_{\beta,\varsigma})}&\leq &1\eq
where $\alpha(\beta)$ is defined in Lemma \ref{alpbet} and
\bq
\fo t\geq 0,\qquad p(\alpha(\beta) t)&\df&1+\exp(\alpha(\beta) t) \eq
and for any $\beta,\sigma>0$
\bq
\fo t\geq 0,\qquad
\vvv  \discret{P}_t^{(\beta,\sigma)}\vvv_{\esL^2(\mathbfit{v}_{\beta,\sigma})\ri\esL^{p({\mathbfit{\alpha}}(\beta,\sigma) t)}(\mathbfit{v}_{\beta,\sigma})}&\leq &1\eq
where ${\mathbfit{\alpha}}(\beta,\sigma)$ is defined in \eqref{wifabs} and
\bq
\fo t\geq 0,\qquad p({\mathbfit{\alpha}}(\beta,\sigma) t)&\df&1+\exp({\mathbfit{\alpha}}(\beta,\sigma) t) \eq
\par
But due to the difficulty in estimating ${\mathbfit{\alpha}}(\beta,\sigma)$, it is preferable to use Theorem \ref{theo2}
to deduce that
\bq
\fo t\geq 0,\qquad
\vvv  \discret{P}_{t+\ln(1+\frac{1}{\sigma})}^{(\beta,\sigma)}\vvv_{\esL^2(\mathbfit{v}_{\beta,\sigma})\ri\esL^{p(\alpha(\beta) t)}(\mathbfit{v}_{\beta,\sigma})}&\leq &1\eq
\par\me
To end this subsection, let us mention
two other applications of the  \cmirs\    of Proposition \ref{pro4}.
\par\sm
$\bullet$ \textbf{Approximations}: For any $\beta,\varsigma>0$, let $X^{(\beta,\varsigma)}\df (X^{(\beta,\varsigma)}_t)_{t\geq 0}$ (respectively $\wi X^{(\beta,\varsigma)}\df (\wi X^{(\beta,\varsigma)}_t)_{t\geq 0}$) be a Markov process associated to $P^{(\beta,\varsigma)}$ (resp.\ $\discret{P}^{(\beta,\varsigma)}$).
As seen in \cite{self1},   for large $\sigma>0$ the birth and death process $(\discret{X}_t^{(\beta,\sigma\varsigma)})_{t\geq 0}$ provides an isospectral approximation of $(X_t^{(\beta,\varsigma)})_{t\geq 0}$. This is related to the fact that $\bar\alpha(\beta)$ should be close to $\alpha(\beta)$, as suggested by \eqref{suggest}.
\par\sm
$\bullet$ \textbf{Simulations}: For  $\sigma>0,\,x\in \RR_+$ and $t\geq 0$, the random variable $Y\df X^{(\beta,\varsigma)}_{\ln(1+\frac{1}{\varsigma\sigma})+t}$ can be simulated by
first sampling $\wi x$ under the probability $\Lambda_\sigma(x,d\wi x)$, next by simulating $\wi X_t^{(\beta,\varsigma\sigma)}$ starting with $\wi X_0^{(\beta,\varsigma\sigma)}=\wi x$ (comprehensively, this amounts to simulate $\wi X_t^{(\beta,\varsigma\sigma)}$ with the initial distribution $\Lambda_\sigma(x,\cdot)$)
and finally by sampling
$Y$ under the probability $\wi \Lambda_{\varsigma^{-1}+\sigma}( \wi X_t^{(\beta,\varsigma\sigma)},\cdot)$.
\par\sm

\subsection{Degenerate hypoelliptic Ornstein-Uhlenbeck processes}
We now describe a refined version of a interweaving relation  between degenerate and non-degenerate hypoelliptic Ornstein-Uhlenbeck semigroups on $\R^d, d\geq 1$, that was identified in  \cite{PA}. In that paper,  the authors exploit the \cmirs\ to obtain the hypocoercive estimate with explicit constants for the convergence to equilibrium in the weighted Hilbert space of the degenerate  hypoelliptic Ornstein-Uhlenbeck semigroups which are non-normal. Therein, we provide further applications of these \cmirs\ to these degenerate semigroups  including   entropy  and hypercontractivity estimates and  the cut-off phenomena.
To define these semigroups, we let $B$ and $\mathrm{\Gamma}$ be  $d\times d$-matrices with $\sigma(B) \subseteq \{z \in \C; \: \Re(z) > 0\}$ and $\mathrm{\Gamma}$ being  positive semi-definite  such that $\det \mathrm{\Gamma}_t > 0$ for all $t > 0$ where
\begin{equation*}
\mathrm{\Gamma}_t = \int_0^t e^{-sB}\mathrm{\Gamma} e^{-sB^*}ds,
\end{equation*}
and the matrix $B^*$ stands for the adjoint of $B$.
 In particular, this holds when $\mathrm{\Gamma}$ is invertible, which we call the non-degenerate case, although it can happen that $\det \mathrm{\Gamma}_t > 0$, for all $t > 0$, with $\det \mathrm{\Gamma} = 0$, which we call the degenerate case. An equivalent condition to $\det \mathrm{\Gamma}_t > 0$ for all $t > 0$  is that $\ker{\mathrm{\Gamma}}$, the kernel of $\mathrm{\Gamma}$, does not contain any invariant subspace of $B^*$. Under these assumptions on $(\mathrm{\Gamma},B)$, the hypoelliptic Ornstein-Uhlenbeck semigroup $P$ admits an unique invariant measure which is the following gaussian distribution
\begin{equation*}
\rho_{\mathrm{\Gamma}_\infty}(d\mathbf{x}) = \frac{e^{-\inn{\mathrm{\Gamma}_\infty^{-1} \textbf{x}}{\textbf{x}}/2}}{\sqrt{(2\pi)^{d}\det \mathrm{\Gamma}_\infty}}d\mathbf{x}, \: \textbf{x} \in \R^d,
\end{equation*}
with $\mathrm{\Gamma}_\infty = \int_0^\infty e^{-tB}\mathrm{\Gamma} e^{-tB^*}ds$ and $\inn{\cdot}{\cdot}$ denotes the Euclidean inner product in $\R^d$.
 $P$ extends to a contraction semigroup on the weighted Hilbert space $\esL^2(\rho_\infty)$.
We also recall that the generator of the Ornstein-Uhlenbeck semigroup $P = (e^{-t\A})_{t \geq 0}$ acts on suitable functions $f$ via
\begin{equation*}
\A [f](\mathbf{x}) = \frac{1}{2}\sum_{i,j=1}^d \gamma_{ij} \partial_i\partial_j f(\textbf{x}) - \sum_{i,j=1}^d b_{ij}x_j \partial_i f(\textbf{x}) =  \frac{1}{2} \trace(\mathrm{\Gamma} \nabla^2)f(\textbf{x}) - \inn{B\textbf{x}}{\nabla}f(\textbf{x}), \quad \textbf{x} \in \R^d,
\end{equation*}
and the condition $\det \mathrm{\Gamma}_t > 0$, for all $t > 0$, is equivalent to the hypoellipticity of $\frac{\partial}{\partial t} + \A$ in the $d+1$ variables $(t,x_1,\ldots,x_d$), hence the terminology. In Metafune, Pallara and Priola \cite[Theorem 3.1]{metafune:2002} (see also Bogatchev \cite{bogachev:2018} and Aleman and Viola \cite{MR3404015})
it was shown that the spectrum of $\A$ in $\esL^2(\rho_{\mathrm{\Gamma}_\infty})$ is entirely determined by the one of the matrix $B$, specifically that, writing $\N = \{0,1,2,\ldots\}$, $\sigma(\A) = \left\lbrace \sum_{i=1}^r k_i b_i; \ k_i \in \N \right\rbrace$, where $b_1,\ldots,b_r$ are the distinct eigenvalues of $B$. Hence, in particular, the spectral gap of $\A$ is ${\lambda}_1=b_\wedge$  as the smallest eigenvalue of $\frac{1}{2}(B+B^*)$. Next, we denote by $\kappa(V)$  the condition number of any invertible matrix $V$, and note that if $V$ is positive-definite then $\kappa(V) = v_\vee/v_\wedge $, where $v_\vee, v_\wedge  > 0$ are the largest and smallest eigenvalues of $V$, respectively.  In the following we write, for a vector $\bm{\alpha} \in \R^d$, $D_{\bm{\alpha}}$ for the diagonal matrix with diagonal entries given by $\bm{\alpha}$. 

\begin{pro}\label{Prop:OU}
Let $P$ be a (possibly) degenerate hypoelliptic Ornstein-Uhlenbeck semigroup associated to $(\mathrm{\Gamma},B)$, that is  $\ker{\mathrm{\Gamma}}$ does not contain any invariant subspace of $B^*$. Suppose that $B$ is diagonalizable with similarity matrix $V,$ and that $\sigma(B) \subseteq (0,\infty)$, that is $VBV^{-1} = D_{\bm{b}}$, where $\bm{b} \in \R^d$ is the vector of eigenvalues of $B$ with $b_i > 0$  for all $i \in \{1,\ldots,d\}$ and we set
\begin{equation*}
\alpha_i = \gamma_{\wedge,\infty} e^{\frac{b_i}{b_\wedge} \log \kappa\left(V\mathrm{\Gamma}_{\infty}V^*\right)} \quad \text{and} \quad \delta_i = {\gamma}_{\infty}
\end{equation*}
where $\gamma_{\wedge,\infty}$ (resp.~$b_\wedge$) is the smallest eigenvalues of $V\mathrm{\Gamma}_{\infty}V^*$ (resp.~$B$).
Then, there exists a non-degenerate hypoelliptic  Ornstein-Uhlenbeck semigroup $\tilde{P}$ associated to $(D_{\bm{\alpha}+2\bm{b}},D_{\bm{b}})$, self-adjoint on $\esL^2(\tilde{\rho}_{D_{\bm{\alpha}}})$, such that
\begin{equation*}
   P \stackrel{\mathbf{t}}{\leftrightsquigarrow} \tilde{P} 
\end{equation*}
where $\mathbf{t} = \frac{1}{b_\wedge} \log \kappa(V \mathrm{\Gamma}_\infty V^*)$,  $\Lambda: \esL^2(\rho_{D_{\bm{\alpha}}}) \rightarrow \esL^2(\rho_{\mathrm{\Gamma}_\infty})$ and $\tilde{\Lambda} : \esL^2(\rho_{\mathrm{\Gamma}_\infty}) \rightarrow \esL^2(\rho_{D_{\bm{\alpha}}})$ are bounded and one-to-one Markov operators defined respectively by
\begin{equation}
\label{eq:Lambda-adjoint}
\Lambda f( \mathbf  x) =  f \ast \rho_{D^{(\bm{\alpha})}}(V\mathbf{x})  \textrm{ and } \tilde\Lambda f(\mathbf{x}) = \frac{1}{\rho_{D^{(\bm{\alpha})}}(\mathbf{x})} ((f_V \ast \rho_{D^{(\bm{\delta})}})   \rho_{D^{(\bm{\delta})}}) \ast \rho_{D_{{\bm{\alpha}}-\bm{\delta}}}(\mathbf{x}), \quad \mathbf{x} \in\R^d,
\end{equation}
where  $\ast$ denotes the additive convolution operator,  for $\bm{a} \in \R^d$, $D^{(\bm{a})}=D_{\bm{a}}-V \mathrm{\Gamma} V^*$ and $f_V(\mathbf  x)=f(V^{-1}\mathbf  x)$.
\end{pro}
\proof
First note that the change of coordinates map ${\Phi_V} f(\mathbf{x}) = f(V^{-1}\mathbf{x})$ is a unitary operator from $\esL^2(\rho_{\mathrm{\Gamma}_\infty})$ to $\esL^2(\rho_{\mathrm{\Gamma}_\infty}^{\Phi_V})$, where $\rho_\infty^{\Phi_V}$ denotes the image density of $\rho_\infty$ under ${\Phi_V}$, i.e.~for $\mathbf{x} \in \R^d$, $\rho_\infty^{\Phi_V}(\mathbf{x}) = \frac{1}{|\det V|}\rho_\infty(V^{-1}\mathbf{x})$.
Next, since $B$ is diagonalizable with similarity matrix $V$ we have that $VBV^{-1} = D_{\mathbf{b}}$, where $\mathbf{b} \in \R^d$ is the vector of eigenvalues of $B$ with $b_i > 0$ for all $i=1,\ldots,d$. Under this change of coordinates, $(\mathrm{\Gamma},B)$ gets mapped to $(V\mathrm{\Gamma}V^*,D_{\mathbf{b}})$ and a simple calculation shows that $\mathrm{\Gamma}_\infty$ then gets mapped to $V\mathrm{\Gamma}_\infty V^*$.  Hence if we prove the desired result for the Ornstein-Uhlenbeck semigroup $\overline{P}$ associated to $(V\mathrm{\Gamma}V^*,D_{\mathbf{b}})$ then, since $P_t = {\Phi_V}^{-1} \overline{P}_t {\Phi_V}$ we get, by Theorem \ref{thm:transp} and the unitary property of $\Phi_V$, that the claims hold for the Ornstein-Uhlenbeck semigroup $P$ associated to $(\mathrm{\Gamma},B)$. From \cite[Proposition 4.2]{PA}, we know that $\cont{\wi P} \stackrel{\mathbf{t}}{\looparrowleft} \overline{P}$ where $\cont{\wi P}$ is the  Ornstein-Uhlenbeck semigroup associated to $(D_{\bm{\alpha}+2\mathbf{b}},D_{\mathbf{b}})$ which is self-adjoint on $\esL^2(\rho_{D_{\bm{\alpha}}})$, hence non-degenerate and the operators $\Lambda$ and  $\tilde\Lambda$ are quasi-affinities on the appropriate weighted  $\esL^2$ spaces. In particular, they are both one-to-one and hence the \cmir\ is symmetric by Theorem \ref{thm:ref-cmir} which completes the proof with another application of Theorem \ref{thm:transp}.
\wwtbp\par

We proceed by providing some by-products of this \cmir. First, we recall that in \cite[Theorem 3.1]{PA}, the following hypocoercive estimate was given, for any $f \in \esL^2(\rho_{\mathrm{\Gamma}_\infty})$,
\begin{equation*}
\fo t\geq 0,\qquad \textrm{Var}_{\rho_{\mathrm{\Gamma}_\infty}}\left(P_t f\right) \leq \kappa(V \mathrm{\Gamma}_\infty V^*) \exp(-2b_\wedge t) \textrm{Var}_{\rho_{\mathrm{\Gamma}_\infty}}\left(f\right)
\end{equation*}
where $\textrm{Var}_{\rho_{\mathrm{\Gamma}_\infty}}\left(f\right)= \int_{\R^d} (f(\textbf{x})-\rho_{\mathrm{\Gamma}_\infty} f)^2\rho_{\mathrm{\Gamma}_\infty}(\textbf{x})d\textbf{x}.$
We carry on by recalling that, in the one-dimensional case $d=1$,  it is well known that the self-adjoint Ornstein-Uhlenbeck semigroup $\tilde{P}^{(i)}, i=1,\ldots, d,$ associated to $(\alpha_i, b_i)$ and whose generator is given by
\begin{equation*}
\tilde \A_{(i)} [f](x) = -\frac{(\alpha_i+2b_i)^2}{2} f''(x) - b_i x f'(x), \quad x \in \R,
\end{equation*}
satisfies the so-called curvature dimension $CD(b_i,\infty)$ which is equivalent to the strict log-Sobolev inequality with constant $b_i$,  see \cite[Section 2.7.1]{bakry:2014}. Then observing that $\tilde{P}$, defined in Proposition \ref{Prop:OU}, is the product of the $\tilde{P}^{(i)}$'s, that is $\tilde{P}    = \bigotimes_{i=1}^d \tilde{P}^{(i)}$,  we get from the stability of the log-Sobolev inequality under products, see  \cite[Proposition 5.2.7]{bakry:2014}, that $\tilde{P}$ satisfies the strict log-Sobolev inequality with constant $b_\wedge$ the minimum of the log-Sobolev constants.   This yields   the following estimate for the convergence in entropy
 \bqn{EntOU}
 \fo t\geq 0,\qquad \Ent(m_0\tilde P_t\vert\tilde{\rho}_\infty)&\leq & \exp(-b_\wedge t)\Ent(m_0\vert \tilde{\rho}_\infty)\eqn
 valid for any initial distribution $m_0$ on $\RR^d$.
Moreover, resorting again to the famous equivalence between the log-Sobolev inequality  and the hypercontractivity property due to Gross \cite{MR0420249}, we get, writing
\bq
\fo t\geq 0,\qquad \tilde{p}(t)&\df&1+\exp(b_\wedge t) \eq
\par
that
\bqn{hypOU}
\fo t\geq 0,\qquad
\vvv  \tilde{P}_{t}\vvv_{\esL^2(\rho_{\infty})\ri\esL^{\tilde{p} (t)}(\rho_{\infty})}&\leq &1\eqn
\par\me
We emphasize that the extension of such estimates to  degenerate hypoelliptic Ornstein-Uhlenbeck semigroup $P$  have met with resistance so far  due to the fact that  $P$ is non-self-adjoint (even non-normal) on $\esL^2(\rho_\infty)$, see \cite[Lemma 3.3]{ottobre:2015}. However, the \cmir\ described in Proposition \ref{Prop:OU} combined with the theorems \ref{theo1} and \ref{theo2}  enable us to obtain the following.

\begin{cor}\label{cor:OU}
Let $P$ be the degenerate hypoelliptic Ornstein-Uhlenbeck semigroup  as defined in Proposition \ref{Prop:OU}. Then,
for any initial distribution $m_0$ on $\RR^d$,
 we have
 \bqn{EntOU}
 \fo t\geq 0,\qquad \Ent(m_0 P_{t}\vert\rho_{\mathrm{\Gamma}_\infty})&\leq & \kappa(V \mathrm{\Gamma}_\infty V^*) \exp(-b_\wedge t)\Ent(m_0\vert \rho_{\mathrm{\Gamma}_\infty})\eqn
and
\bqn{hypOU}
\fo t\geq 0,\qquad
\vvv  {P}_{t+\mathbf{t}}\vvv_{\esL^2(\rho_{\mathrm{\Gamma}_\infty})\ri\esL^{\tilde{p} (t)}(\rho_{\mathrm{\Gamma}_\infty})}&\leq &1\eqn
\par\me
\end{cor}
 We mention that Arnold and Erb \cite{arnold:2014} have  obtained hypocoercivity estimate of the form \eqref{EntOU}, under our assumptions, with exponential rate given by the spectral gap $b_\wedge$ and that Arnold et al.~\cite{arnold:2018} and Monmarch\'e \cite{monmarche:2019} have proved hypocoercivity with exponential rate $b_\wedge$ without assuming that $B$ is diagonalizable. However, in contrast to these existing results, we are able to explicitly identify the constant in front of the exponential, i.e.~$\kappa(V\mathrm{\Gamma}_\infty V^*)$, in terms of the initial data $\mathrm{\Gamma}$ and $B$. Note that, in particular, if $B$ is symmetric then $V$ is unitary and $\kappa(V\mathrm{\Gamma}_\infty V^*) = \kappa(\mathrm{\Gamma}_\infty)$.  However we are not aware of  results regarding the hypercontractivity estimates.

We now turn  to another application of interweaving which allows to identify the cut-off phenomena for degenerate hypoelliptic Ornstein-Uhlenbeck semigroups.
To this end, let $\alpha\df (\alpha_1, \alpha_2, ..., \alpha_d)$ and $b\df(b_1,b_2, ..., b_d)$ be vectors from $(0,+\iy)^d$.
Denote $b_\wedge\df \min (b_l, \, l\in \lin d\rin)$ and
for any $n\in\NN$,  $\bm{\alpha}^{(n)}\df(\alpha,\ldots,\alpha) \in \RR^{dn}$ and $\bm{b}^{(n)}\df(b,\ldots,b)\in \RR^{dn}$.
 Consider the family of semigroups $(\tilde P^{(n)})_{n\in \ZZ_+}$  associated for each $n \in \ZZ_+$ to $(D_{\bm{\alpha}^{(n)}+2\bm{b}^{(n)}},D_{\bm{b}^{(n)}})$.
 Lachaud \cite{Lachaud}
 has shown that this family
has a cut-off at the time
\bqn{eq:cut-off-t}
  t^{(n)} &\df& \frac{\log n }{2 b_\wedge}
\eqn
(more precisely,  Lachaud \cite{Lachaud} has only considered the case $d=1$, but her arguments extend to any $d\in\NN$, see also Barrera, Lachaud and  Ycart \cite{MR2260742}).
\par We have the following generalization.
\begin{cor}\label{cor:OU}
For any $n \in \NN$, let $P^{(n)}$ be the degenerate hypoelliptic Ornstein-Uhlenbeck semigroup in $\RR^{dn}$  as defined in Proposition \ref{Prop:OU} and associated to some $(\mathrm{\Gamma}^{(n)},B^{(n)})$, with $\kappa(\mathrm{\Gamma}^{(n)}_\infty)$
satisfying
\bq
\lim_{n\ri\iy} \frac{\log(n)}{\kappa(\mathrm{\Gamma}^{(n)}_\infty)}&=&+\iy\eq
and $\bm{\alpha}^{(n)}$ as above $\bm{b}^{(n)}$ as above. Then,  the family $(P^{(n)})_{n\in \NN}$ has a cut-off  at the times $(t^{(n)})_{n\in\NN}$ defined in \eqref{eq:cut-off-t}.
\end{cor}
\proof Under the conditions of the claim, we easily check that Proposition \ref{Prop:OU} entails that for each  $n \in \NN$, $\cont{ P^{(n)}} \stackrel{\mathbf{t}^{(n)}}{\looparrowleft} \wi{P}^{(n)}$ where  $\mathbf{t}^{(n)}= \frac{1}{b_\wedge} \log \kappa( \mathrm{\Gamma}^{(n)}_\infty)$  and $\tilde P^{(n)}$ is the semigroup of the self-adjoint Ornstein-Uhlenbeck process  defined before the corollary.
We conclude the proof by invoking the result of  Lachaud \cite{Lachaud}   recalled before the corollary and  Theorem \ref{theo3}.
\wwtbp\par


To finish this section, let us give a concrete example.
\par
Consider the matrices
\bq
\mathrm{\Gamma}\ \df\ \lt(\begin{array}{cc}0&0\\0&1\end{array}\rt)&\qquad &B\ \df\ \lt(\begin{array}{cc}0&-1\\1/2&2\end{array}\rt)
\eq\par
The corresponding Ornstein-Uhlenbeck is a simple example of a kinetic model: the first coordinate corresponds to the position in $\RR$
of a particle in the quadratic potential $\RR\ni x\mapsto x^2/4$,
 and the second coordinate is the  speed, on which is acting a Brownian motion.
It is a typical instance of a hypoelliptic system.
To see it admits an invariant probability and the existence of $\mathrm{\Gamma}_\iy$, it is sufficient to check that the eigenvalues $b_1,b_2$ of $B$ are positive.
They are indeed the solutions of the second order equation
$X^2-2X+1/2=0$ and we get
\bq
b_1&=&1-1/\sqrt{2}\\
b_2&=&1+1/\sqrt{2}
\eq
\par Let $\mathrm{\Gamma}_\iy$ and $\alpha_1,\,\alpha_2>0$ be as in Proposition  \ref{Prop:OU}.
Denote $b\df(b_1,b_2)$ and $\alpha\df(\alpha_1,\alpha_2)$.
\par
For any $n\in\NN$, introduce the tensorizations
\bq
\mathrm{\Gamma}^{(n)}\ \df\ \lt(\begin{array}{ccccc}\mathrm{\Gamma}&0& 0&\cdots&0\\
0&\mathrm{\Gamma}&0& \cdots&0\\
0&0&\mathrm{\Gamma}&\ddots&0\\
\vdots&\vdots&\ddots& \ddots&\vdots\\
0&0&\dots&0&\mathrm{\Gamma}
\end{array}\rt)&\qquad &B^{(n)}\ \df\ \lt(\begin{array}{ccccc}B&0& 0&\cdots&0\\
0&B&0& \cdots&0\\
0&0&B&\ddots&0\\
\vdots&\vdots&\ddots& \ddots&\vdots\\
0&0&\dots&0&B
\end{array}\rt)
\eq
\par
This block structure implies that for any $n\in\NN$, the Ornstein-Uhlenbeck semigroup  $P^{(n)}$ associated to $(\mathrm{\Gamma}^{(n)},B^{(n)})$ is hypoelliptic  and we have
\bq
\mathrm{\Gamma}_\iy^{(n)}\ \df\ \lt(\begin{array}{ccccc}\mathrm{\Gamma}_\iy&0& 0&\cdots&0\\
0&\mathrm{\Gamma}_\iy&0& \cdots&0\\
0&0&\mathrm{\Gamma}_\iy&\ddots&0\\
\vdots&\vdots&\ddots& \ddots&\vdots\\
0&0&\dots&0&\mathrm{\Gamma}_\iy
\end{array}\rt)
\eq\par
In particular $\kappa(\mathrm{\Gamma}_\iy^{(n)})$ does not depend on $n\in\NN$ and
$\bm{\alpha}^{(n)}=(\alpha,\ldots,\alpha) \in \RR^{dn}$ and $\bm{b}^{(n)}=(b,\ldots,b)\in \RR^{dn}$.
\par
It follows from Corollary \ref{cor:OU} that  the family $(P^{(n)})_{n\in \NN}$
has a cut-off  at the times $(\log(n)/(2b_1))_{n\in\NN}$.

\section{Random  warm-up time examples} \label{sec:rwt}

In this section, we present several examples of \cmirs\ for which the warm-up time is a positive random variable. This includes the family of Laguerre and  Jacobi processes and examples of  Subsection~\ref{cadLp} that  are
extended in various directions either by playing with the underlying parameters or by pertubating their generator by  a non-local component, that is by adding  jumps in their dynamics.   We also describe several interesting applications of \cmirs\ in these contexts.

\subsection{Diffusive Laguerre operators}\label{dLo}

The classical Laguerre generators $L_{\beta,\sigma}$, for $\beta,\sigma>0$, were recalled in Subsection \ref{cadLp}.
Here we will drop the second parameter $\sigma>0$, since we are more interested in the parameter $\beta>0$:
we would like to counter the bad
behavior of the logarithmic Sobolev constant for small $\beta>0$ via \cmirs, in the spirit of what we have done for the two-point state space in Subsection \ref{2points}.
Namely we are looking for \cmirs\  between Laguerre semigroups with different parameters $\beta>0$.\par\me
For any $\beta>0$, we write simply $L_{\beta}\df L_{\beta,1}$, $\nu_{\beta}\df \nu_{\beta,1}$ and  $P^{(\beta)}\df P^{(\beta,1)}$, with the notations of Subsection \ref{cadLp}.
For any $\beta,\ptbeta>0$, consider the Markov kernel $\LaLa $ from $\RR_+$ to $\RR_+$ corresponding to the multiplication by
a Beta random variable of parameters $\ptbeta$ and $\deltabeta$, namely for any $f\in\esB(\RR_+)$, the set of bounded measurable mappings on $\RR_+$,
\bq
\fo x\in\RR_+,\qquad \LaLa [f](x)&\df&\frac{\Gamma(\gdbeta)}{\Gamma(\deltabeta)\Gamma(\ptbeta)} \int_0^1 f(rx) r^{\ptbeta-1}(1-r)^{\deltabeta-1}\, dr\eq
\par
Its interest for us, is that according to Patie and Savov \cite{Patie-Savov-GeL} 
we have the intertwining relation
\bq
\fo \beta>\ptbeta>0,\,\fo t\geq 0,\qquad P_t^{(\gdbeta)}\LaLa &=&\LaLa P_t^{(\ptbeta)}\eq
where the products are understood as the compositions of Markov kernels. They can also be seen as compositions of operators acting on $\esL^2$-spaces and we have the
following commuting diagram for any $\beta>\ptbeta>0$ and $ t\geq 0$:
\begin{figure}[H]\centering
\begin{tikzcd}[scale=2]
\esL^2(\nu_\gdbeta)\arrow[r, "P_t^{(\gdbeta)}"] \arrow[d, "\LaLa "']
& \esL^2(\nu_\gdbeta)\arrow[d, "\LaLa " ]  \\
\esL^2(\nu_{\ptbeta})\arrow[r, "P_t^{(\ptbeta)}" ]
& \esL^2(\nu_{\ptbeta})
\end{tikzcd}
\caption{Intertwining relation between $P_t^{(\gdbeta)}$ and $P_t^{(\ptbeta)}$}\label{fig1b}
\end{figure}
\par
To get an intertwining relation in the reverse direction, we pass to the adjoint relations, taking into account that
$P_t^{(\gdbeta)}$ and $P_t^{(\ptbeta)}$ are self-adjoint in $\esL^2(\nu_\gdbeta)$ and
$\esL^2(\nu_{\ptbeta})$ respectively:\par
\begin{figure}[H]\centering
\begin{tikzcd}[scale=2]
\esL^2(\nu_{\ptbeta})\arrow[r, "P_t^{(\ptbeta)}"] \arrow[d, "\LaLad"']
& \esL^2(\nu_{\ptbeta})\arrow[d, "\LaLad" ]  \\
\esL^2(\nu_{\gdbeta})\arrow[r, "P_t^{(\gdbeta)}" ]
& \esL^2(\nu_{\gdbeta})
\end{tikzcd}
\caption{Intertwining relation between  $P_t^{(\ptbeta)}$ and $P_t^{(\gdbeta)}$}\label{fig2b}
\end{figure}
\noindent
where $\LaLad\st \esL^2(\nu_\gdbeta)\ri \esL^2(\nu_{\ptbeta})$
is the adjoint operator of
$\LaLa \st \esL^2(\nu_{\ptbeta})\ri \esL^2(\nu_{\gdbeta})$.\par
Since $\nu_\gdbeta$ and $\nu_{\ptbeta}$ are both probability measures, it is known a priori that $\LaLad$
corresponds to a Markov kernel. Let us compute it more precisely:
\begin{lem}
We have for any $\beta,\ptbeta>0$ and any $g\in\esB(\RR_+)$,
\bq
\fo x\in\RR_+,\qquad \LaLad[g](x)&=&\frac{x^{\deltabeta}}{\Gamma(\deltabeta)} \int_0^{+\iy} g((1+s)x)\, s^{\deltabeta} \exp(-sx)\, ds\eq
\end{lem}
\proof
For any $f,g\in \esB(\RR_+)$,
we compute
\bq
\nu_{\gdbeta}[g\LaLa[f]]&=&
\frac{\Gamma(\gdbeta)}{\Gamma(\deltabeta)\Gamma(\ptbeta)}\int_{\RR_+} g(x)\lt( \int_0^1 f(rx) r^{\ptbeta-1}(1-r)^{\deltabeta-1}\, dr\rt)
\frac{x^{\gdbeta-1}\exp(-x)}{\Gamma(\gdbeta)}\, dx\\
&=&\int_0^1\lt( \int_{\RR_+} g(x) f(rx)x^{\gdbeta-1}\exp(-x)\, dx \rt)
r^{\ptbeta-1}(1-r)^{\deltabeta-1}\, dr\\
&=&
\int_0^1\lt(r^{-(\gdbeta)} \int_{\RR_+} g(x/r) f(x)x^{\gdbeta-1}\exp(-x/r)\, dx \rt)
r^{\ptbeta-1}(1-r)^{\deltabeta-1}\, dr\\
&=&
\int_{\RR_+}f(x)\lt(x^{\deltabeta} \int_0^1 g(x/r)\, r^{\ptbeta-\gdbeta-1} (1-r)^{\deltabeta-1}\exp(-x(1/r-1))\, dr\rt)x^{\ptbeta-1} \exp(-x) \,dx
\eq
Since the last expression must be equal to $\Gamma(\deltabeta)\Gamma(\ptbeta)\nu_{\ptbeta}[f\Lambda^*_{\gdbeta,\ptbeta}[g]]$, for any $f\in \esB(\RR_+)$, we obtain
\bq
\fo x\in\RR_+,\qquad \LaLad[g](x)&=&\frac{x^{\deltabeta}}{\Gamma(\deltabeta)} \int_0^1 g(x/r)\,  \frac1{r^2}\lt(\frac{1-r}{r}\rt)^{\deltabeta-1}\exp(-x(1-r)/r)\, dr
\eq
and
we deduce the announced result via the change of variable $s=(1-r)/r$.
\wwtbp\par
To get a c.mi.r., let us compute  the Markov kernel $\LaLa \LaLad$. Following the argumentation of Remark \ref{1errems}(f), we know a priori that $\LaLa \LaLad$
commutes with the $P_t^{(\gdbeta)}$, for all $t\geq 0$. Since $L_\gdbeta$ is diagonalizable in $\esL^2(\nu_\gdbeta)$ and all its eigenvalues are non-positive and simple,
it follows from functional calculus that $\LaLa \LaLad$ is of the form $F(-L_\gdbeta)$, where $F\st \RR_+\ri \RR$ is a measurable mapping.
Here is its explicit formula:
\begin{pro}\label{complmonot}
For any $\beta,\ptbeta>0$, we have \bqn{LLFL}
\LaLa \LaLad&=&F_{\beta_{\varepsilon}}(-L_\gdbeta)\eqn
with
\bqn{Fbb}
\fo u\in\RR_+,\qquad
F_{\beta_{\varepsilon}}(u)&\df&  \int_{0}^{\infty}e^{-us}\P(\tau^{(\beta_{\varepsilon})} \in ds)=\frac{\Gamma(\beta+\varepsilon)\Gamma(u+\varepsilon)}{\Gamma(\varepsilon)\Gamma(u+\beta+\varepsilon)} 
\eqn
and
\bqn{pibetabeta}
\fo s\geq 0,\qquad \P(\tau^{(\beta_{\varepsilon})} \in ds)&\df& \frac{\Gamma(\gdbeta)}{\Gamma(\deltabeta)\Gamma(\ptbeta)}\exp(-\ptbeta s)(1-\exp(-s))^{\deltabeta-1}\, ds\eqn
Similarly, we  have, still for $\beta,\ptbeta>0$,
 \bq
\LaLad\LaLa &=&F_{\beta_{\varepsilon}}(-L_{\ptbeta})\eq
\end{pro}
\proof
It is well-known (see e.g.\ the book of Szeg\"{o} \cite{MR0372517}) that the spectrum of $-L_\gdbeta$ is $\ZZ_+$ and for each eigenvalue $n\in\ZZ_+$,
an associated eigenvector is the Laguerre polynomial $\mathcal{L}_n^{(\gdbeta)}$ of degree $n$.
It follows that to prove
\eqref{LLFL},
it is sufficient to show that for any $n\in\ZZ_+$,
we have
\bq
\LaLa \LaLad [\mathcal{L}_n^{(\gdbeta)}] &=&F_{\beta_{\varepsilon}}(n)[\mathcal{L}_n^{(\gdbeta)}]\eq
From the commutation of $\LaLa \LaLad$
 with the $P_t^{(\gdbeta)}$ for all $t\geq 0$, we know a priori that the l.h.s.\ is proportional to $\mathcal{L}_n^{(\gdbeta)}$.
 Thus, denoting $p_n\st\RR_+\ni x\mapsto x^n$,
 the monomial of degree $n$, it is sufficient to check that
 $\LaLa \LaLad[p_n]$ is equal to  $F_{\beta_{\varepsilon}}(n)p_n$, up to a polynomial of degree $n-1$.
 This operation can be decomposed into two similar sub-tasks. Indeed from
 Figure \ref{fig1b}  we deduce that for any $t\geq 0$,
 \bq P_t^{(\gdbeta)}\LaLa [\mathcal{L}_n^{(\ptbeta)}]\ =\ \LaLa P_t^{(\ptbeta)}[\mathcal{L}_n^{(\ptbeta)}]
 \ =\ \exp(-nt)\LaLa [\mathcal{L}_n^{(\ptbeta)}]\eq
 namely $\LaLa [\mathcal{L}_n^{(\ptbeta)}]$ is proportional to $\mathcal{L}_n^{(\gdbeta)}$.
 So let $\wi F_{\beta_{\varepsilon}}(n)\in\RR$ be such that $\LaLa [p_n]$ is equal to $\wi F_{\beta_{\varepsilon}}(n)p_n$, up to a polynomial of degree $n-1$.
 Similarly,  taking into account  Figure \ref{fig2b}, there exists $\wit F_{\beta_{\varepsilon}}(n)\in \RR$ such that
 $\LaLa ^*[p_n]$ is equal to $\wit F_{\beta_{\varepsilon}}(n)p_n$, up to a polynomial of degree $n-1$.
 It follows that $F_{\beta_{\varepsilon}}(n)=\wi F_{\beta_{\varepsilon}}(n)\wit F_{\beta_{\varepsilon}}(n)$ and we just need to compute $\wi F_{\beta_{\varepsilon}}(n)$ and $\wit F_{\beta_{\varepsilon}}(n)$.
 Let us start with $\wi F_{\beta_{\varepsilon}}(n)$. We have  for any $x\in\RR_+$,
 \bqn{F1n}
\nonumber \LaLa [p_n](x)&=&
 \frac{\Gamma(\gdbeta)}{\Gamma(\deltabeta)\Gamma(\ptbeta)} \int_0^1 (rx)^n r^{\ptbeta-1}(1-r)^{\deltabeta-1}\, dr\\
 &=&
 \frac{\Gamma(\gdbeta)}{\Gamma(\deltabeta)\Gamma(\ptbeta)} \lt(\int_0^1  r^{n+\ptbeta-1}(1-r)^{\deltabeta-1}\, dr\rt) x^n\\
\nonumber &=&
 \frac{\Gamma(\gdbeta)}{\Gamma(\deltabeta)\Gamma(\ptbeta)} \frac{\Gamma(n+\ptbeta)\Gamma(\deltabeta)}{\Gamma(n+\gdbeta)}p_n(x)
 \eqn
 and thus
 \bqn{wFbwb}
 \wi F_{\beta_{\varepsilon}}(n)\ =\ \frac{\Gamma(\gdbeta)\Gamma(n+\ptbeta)}{\Gamma(\ptbeta)\Gamma(n+\gdbeta)}\ =\ \frac{(n+\ptbeta-1)(n+\ptbeta-2)\cdots \ptbeta}{(n+\gdbeta-1)(n+\gdbeta-2)\cdots \gdbeta}.
  \eqn
 On the other hand, for $\wit F_{\beta_{\varepsilon}}(n)$, we have  for any $x\in\RR_+$,
 \bq\LaLad[p_n](x)&=&\frac{x^{\deltabeta}}{\Gamma(\deltabeta)}  \int_0^{+\iy} ((1+s)x)^n\, s^{\deltabeta-1} \exp(-sx)\, ds
 \\
 &=&\frac{x^{\deltabeta}}{\Gamma(\deltabeta)} \lt( \int_0^{+\iy} (1+s)^{n}\, s^{\deltabeta-1} \exp(-sx)\, ds \rt)x^n\\
 &=& \frac{1}{\Gamma(\deltabeta)} \lt( \int_0^{+\iy} (1+s/x)^{n}\, s^{\deltabeta-1} \exp(-s)\, ds \rt)x^n\\
 &=&\sum_{m=0}^{n}\binom{n}{m} \frac{1}{\Gamma(\deltabeta)} \lt( \int_0^{+\iy} s^m\, s^{\deltabeta-1} \exp(-s)\, ds \rt)x^{n-m}.
 \eq
It follows that
\bq
\wit F_{\beta_{\varepsilon}}(n)&=&\binom{n}{0} \frac{1}{\Gamma(\deltabeta)}  \int_0^{+\iy} \, s^{\deltabeta-1} \exp(-s)\, ds
=\frac{1}{\Gamma(\deltabeta)} \Gamma(\deltabeta)
=1\eq
Thus we get that for all $n\in\ZZ_+$, $F_{\beta_{\varepsilon}}(n)=\wi F_{\beta_{\varepsilon}}(n)$.
Coming back to \eqref{F1n}, it appears that for any $n\in\ZZ_+$,
\bq
F_{\beta_{\varepsilon}}(n)&=& \frac{\Gamma(\gdbeta)}{\Gamma(\deltabeta)\Gamma(\ptbeta)} \int_0^1  r^{n+\ptbeta-1}(1-r)^{\deltabeta-1}\, dr\\
&=& \frac{\Gamma(\gdbeta)}{\Gamma(\deltabeta)\Gamma(\ptbeta)}\int_0^{+\iy}
\exp(-ns)\exp(-\ptbeta s)(1-\exp(-s))^{\deltabeta-1}\, ds\eq
where we considered the change of variable $r=\exp(-s)$. It justifies \eqref{LLFL}.
\par
The last assertion of the proposition is proven similarly, or by applying the $\esL^2$-version of Proposition \ref{symmetric}: $\LaLa$ is one-to-one, since it transforms
the orthogonal basis $(\mathcal{L}_n^{(\ptbeta)})_{n\in\ZZ_+}$ of $\esL^2(\nu_{\ptbeta})$ into an orthogonal basis of $\esL^2(\nu_{\gdbeta})$:
\bq
\fo n\in\ZZ_+,\qquad \LaLa[\mathcal{L}_n^{(\ptbeta)}]&=&\wi F_{\beta_{\varepsilon}}(n)\mathcal{L}_n^{(\gdbeta)}\eq
where $\wi F_{\beta_{\varepsilon}}(n)>0$ is given in \eqref{wFbwb}.
\wwtbp
\par
Thus we have shown the symmetric c.m.i.r.
between $P^{(\gdbeta)}_t$ and $P^{(\ptbeta)}_t$ described in the following Figure \ref{fig5}, for any $\gdbeta>\ptbeta>0$ and $t\geq 0$:
\begin{figure}[H]\centering
\begin{tikzcd}[scale=2]
\esL^2(\nu_{\gdbeta})\arrow[r, "P_t^{(\gdbeta)}"] \arrow[d, "\LaLa "] \arrow[dd, bend right=50 , "P_{\tau^{(\beta_{\varepsilon})}}^{(\gdbeta)}"']
& \esL^2(\nu_{\gdbeta})\arrow[d, "\LaLa "' ]  \arrow[dd, bend left=50 , "P_{\tau^{(\beta_{\varepsilon})}}^{(\gdbeta)}"]\\
\esL^2(\nu_{\ptbeta})\arrow[dd, bend right=50 , "P_{\tau^{(\beta_{\varepsilon})}}^{(\ptbeta)}"']\arrow[r, "P_t^{(\ptbeta)}" ]\arrow[d, "\LaLad"]
& \esL^2(\nu_{\ptbeta})\arrow[dd, bend left=50 , "P_{\tau^{(\beta_{\varepsilon})}}^{(\ptbeta)}"]\arrow[d, "\LaLad"' ]\\
\esL^2(\nu_{\gdbeta}) \arrow[d, "\LaLa "] \arrow[r, "P_t^{(\gdbeta)}"]
& \esL^2(\nu_{\gdbeta}) \arrow[d, "\LaLa "'] \\
\esL^2(\nu_{\ptbeta}) \arrow[r, "P^{(\ptbeta)}_{t}"]
& \esL^2(\nu_{\ptbeta})
\end{tikzcd}
\caption{interweaving relations  between $P_t^{(\gdbeta)}$ and $P_t^{(\ptbeta)}$}\label{fig5}
\end{figure}
\par
Since we are interested in the behavior for small shape parameter, let us denote for $\gdbeta\in(0,1/2)$,
$\tau_\gdbeta\df\pi_{1/2,\gdbeta}$.
We deduce the following bound
from Theorem \ref{theo1} and from the fact that $\alpha(1/2)=1$:
\begin{cor}
For any $\ptbeta \in(0,1/2)$ and any $m_0\in\cP((0,+\iy))$,
\bqn{hypo2}
\fo t\geq 0,\qquad  \Ent(m_{0}P_{t+\tau^{(\beta_\ptbeta)}}^{(\ptbeta)}\vert \nu_\ptbeta)&\leq & \exp(-t)\Ent(m_0\vert\nu_\ptbeta)\eqn
\end{cor}
Recall the estimate directly obtained by applying the logarithmic Sobolev inequality
satisfied by the generator $L_\ptbeta$ for any $\ptbeta\in(0,1/2)$:
\bqn{hypo3}
\fo m_0\in\cP((0,+\iy)),\,\fo t\geq 0,\qquad  \Ent(m_{0}P_{t}^{(\ptbeta)}\vert \nu_\ptbeta)&\leq & \exp(-\alpha(\ptbeta)t)\Ent(m_0\vert\nu_\ptbeta)\eqn
(where $\alpha(\ptbeta)$ is defined in \eqref{alpbet}).
The bounds \eqref{hypo2} and \eqref{hypo3} are not directly comparable, since they concern different distributions, namely $m_{0}P_{t+\tau^{(\beta_\ptbeta)}}^{(\ptbeta)}$
and $m_{0}P_{t}^{(\ptbeta)}$ and the former is not just a deterministic time translate through $P$ of the latter.
Nevertheless, to highlight the potential advantage of \eqref{hypo2}, let us make the following observation. Let $X^{(\ptbeta)}\df(X^{(\ptbeta)}_t)_{t\geq 0}$ be a diffusion process associated to the Markov semigroup $P^{(\ptbeta)}$, with small $\ptbeta\in(0,1/2)$,
starting with $X^{(\ptbeta)}_0$ uniformly distributed over $[0,1]$.
We want to use this trajectory to sample according to $\nu_\ptbeta$, with an accuracy given by $\delta>0$ in the entropy sense.
Relying on \eqref{hypo3}, we consider the position $X^{(\ptbeta)}_{T_1}$ at  the time $T_1\geq 0$ such that
\bq
 \exp(-\alpha(\ptbeta)T_1)\Ent(m_0\vert\nu_\ptbeta)&\leq & \delta
 \eq
 for some $\delta>0$.
 Letting $\ptbeta$ going to $0_+$ and recalling that $\alpha(\ptbeta)\sim 4/\ln (1/\ptbeta) $,
 we easily compute  that
 \bq
 \Ent(m_0\vert\nu_\ptbeta)&=&\int_0^1\ln(\Gamma(\ptbeta)x^{1-\ptbeta}\exp(x))\, dx
 =\ln(\Gamma(\ptbeta))+(1-\ptbeta)\int_0^1\ln(x)\,dx+\int_0^1x\, dx\\
 &=&\ln(\Gamma(\ptbeta))+\ptbeta-1/2
 \sim \ln(\Gamma(\ptbeta))\\
 &\sim& \ln(1/\ptbeta)\eq
 So we get that
 \bq
 T_1&\simeq& \ln(1/\ptbeta)\ln(\ln(1/\ptbeta)/\delta)/4
 =\ln(1/\ptbeta)(\ln(\ln(1/\ptbeta)+\ln(1/\delta))/4
 \eq
 Relying on \eqref{hypo2},
 we consider the position $X^{(\ptbeta)}_{T_2+T_3}$, where  $T_2$ is independent from $X^{(\ptbeta)}$ and has the same law than $\tau_\ptbeta$
 and $T_3\geq 0$
is such that
\bq
 \exp(-T_3)\Ent(m_0\vert\nu_\ptbeta)&\leq & \delta\eq
 namely
 \bq
 T_3&\simeq& \ln(\ln(1/\ptbeta)/\delta)\eq
 To get a rough idea of $T_2$, let us compute its expectation, as $\ptbeta$ goes to zero:
 \bq
 \EE[T_2]&=&\int_0^{+\iy} s\, \P(\tau^{(\beta_{\varepsilon})} \in ds)= \frac{\Gamma(1/2)}{\Gamma(1/2-\ptbeta)\Gamma(\ptbeta)}\int_0^{+\iy} s\exp(-\ptbeta s)(1-\exp(-s))^{-\ptbeta-1/2}\, ds\\
 &=& -\frac{\Gamma(1/2)}{\Gamma(1/2-\ptbeta)\Gamma(\ptbeta)}\int_0^1\ln( r) r^{\ptbeta-1}(1-r)^{-\ptbeta-1/2}\, dr\\
 &\sim& -\frac{1}{\Gamma(\ptbeta)}\int_0^1 \ln (r) r^{\ptbeta-1}\, dr
 =\frac{1}{\Gamma(\ptbeta+1)}\int_0^1r^{\ptbeta-1}\, dr
  =\frac{1}{\ptbeta\Gamma(\ptbeta+1)}\\
  &\sim& \frac1\ptbeta
 \eq
 (where an integration by parts was used for the fourth equality), and thus
 \bq
 \EE[T_2]+T_3&\simeq&  \frac1\ptbeta+ \ln(1/\delta)\eq
 \par
 When $\delta>0$ is very small, e.g.\ of order $\exp(-1/\ptbeta)$, the quantity $ \EE[T_2]+T_3$ is much smaller than $T_1$,
 suggesting that the approach based on \eqref{hypo2} is a more effcient sampling procedure.
 \par\me
 Similar observations are also valid for hyperboundedness, as we deduce from Theorem \ref{theo2}:
 \begin{cor}
 For any $\ptbeta\in(0,1/2)$, we have
 \bqn{hyperbound}
\fo t\geq 0,\qquad  \vvv P^{(\ptbeta)}_{t+\tau^{(\beta_\ptbeta)}}\vvv_{\esL^2(\nu_\ptbeta)\ri\esL^{p( t)}(\nu_\ptbeta)}&\leq & 1\eqn
where
\bq
\fo t\geq 0,\qquad p(t)&\df&1+\exp(t) \eq
 \end{cor}
Note that for small $\ptbeta >0$ and large $t\geq 0$, the exponent $ p(t)$ is much larger than $1+\exp(\alpha(\ptbeta)t)$,
the quantity one gets via the
 traditional application of the logarithmic Sobolev associated to $L_\ptbeta$.
 Thus up to waiting a warm-up time variable $\tau^{(\beta_\ptbeta)})$, the
 hyperboundedness estimate \eqref{hyperbound} is more interesting than the usual hypercontractive bound.

\subsection{The Jacobi processes }
We proceed with another important and classical example in the theory of diffusions which is the Jacobi semigroup ${\rm{J}}^{(\beta)}=({\rm{J}}^{(\beta)}_t)_{t\geq 0}$.  Its infinitesimal generator  is defined for a  function $f \in {\rm{C}}^{2}(V)$, the space of twice continuously differentiable functions on $V=[0,1]$, by
\begin{equation}\label{J_infg}
J_{\beta} [f](x) = x(1-x)f''(x) +(\lambda_1 -\beta -\lambda_1 x)f'(x),  \quad x\in [0,1],
\end{equation}
where $ \lambda_1 \geq  2 \beta > 2$ and refer here and below to \cite[Section 5]{CPVS} for a thorough review of the Jacobi semigroup.

It admits as unique invariant measure ${\nu}_{\beta}$, the distribution of a beta $B(\lambda_1,\beta)$ random variable, defined on
$(0,1)$ as
\[ {\nu}_{\beta}(dx)=\frac{\Gamma(\lambda_1)}{\Gamma(\lambda_1-\beta)\Gamma(\beta)}x^{\beta-1}(1-x)^{\lambda_1-\beta- 1}dx, \quad 0<x<1. \]
As a by-product, the H\"older inequality yields that  ${\rm{J}}^{(\beta)}$ extends to a contraction semigroup from the Hilbert space $\esL^2({\nu}_{\beta})$ into itself.
We recall that for any $n\in \mathbb{N}$,
\[ \int_0^{\infty}x^n{\nu}_{\beta}(dx)=\frac{\Gamma(\lambda_1)}{\Gamma(\beta)}\frac{\Gamma(n+\beta)}{\Gamma(n+\lambda_1)}. \]
We say that the Jacobi operator is symmetric when $\lambda_1=  2 \beta$ and, in this case,  we write $\wi{ \rm{J}}=({ \rm{ \wi J}}_t)_{t\geq 0}$ for the symmetric Jacobi semigroup whose infinitesimal generator  is $\wi J = J_{\frac{\lambda_1}{2}} $ that is
\[ \wi J [f](x)  =x(1-x)f''(x) +\frac{\lambda_1}{2}(1- 2x)f'(x), \quad 0<x<1.\]
We remark that, when $\frac{\lambda_1}{2} = n \in \N$, there exists a homeomorphism between $J_\beta$ and the radial part of the Laplace-Beltrami operator on the $n$-sphere, which leads to the curvature-dimension condition $CD(\lambda_1 -1,\lambda_1 )$, see \cite{bakry:2014} for the definition. We deduce from \cite[Proposition 3.6]{CPVS}, choosing in the notation thereout $\mu=\frac{\lambda_1}{2}$ and $\hbar\equiv0$, the following interweaving relation between the symmetric and other Jacobi semigroups.
\begin{pro}
 For any $\lambda_1>2\beta>1$, we have
  \begin{equation*}
   \wi{ \rm{J}} \stackrel{\tau^{(\lambda_1,\beta)}}{\leftrightsquigarrow}{\rm{J}}^{(\beta)}  
\end{equation*}
with
\bqn{Fbb}
\fo u\in\RR_+,\qquad
  \int_{0}^{\infty}e^{-us}\P(\tau_{\phi}^{(\lambda_1,\beta)} \in ds)=\frac{\Gamma(\lambda_1-\beta)\Gamma(\rho(u)+\frac{\lambda_1}{2})}{\Gamma(\rho(u)+\lambda_1-\beta)\Gamma(\frac{\lambda_1}{2})} 
\eqn
where $\rho(u)=\sqrt{u+\frac{(\lambda_1-1)^2}{4}}-\frac{\lambda_1-1}{2}.$
\end{pro}
As a self-adjoint operator $J_\beta$ has nice spectral properties: its spectrum is discrete with simple eigenvalues  given by the set  $(-n(n-1)-\lambda_1 n)_{n\geq0}$.  Moreover, it satisfies certain functional inequalities which give some quantitative rates of convergence to the equilibrium measure ${\nu}_{\beta}$. For instance, from the Poincar\'e inequality for $J_\beta$, see \cite[Chapter 4.2]{bakry:2014}, one gets the following variance decay estimate, valid for any $f \in \esL^2({\nu}_{\beta})$ and $t \geq 0$,
\begin{equation*}
\Var_{{\nu}_{\beta}}({\rm{J}}^{(\beta)}_t [f] ) \leq e^{-2\lambda_1t} \Var_{{\nu}_{\beta}}(f),
\end{equation*}
where  for a measure $\nu$, we have set $\textrm{Var}_{\nu}\left(f\right)=||f-\nu f||^2_{\L^2(\nu)}$. Next, note,  writing \begin{equation}
\bar{J}_{\beta} [f](x) = (1-x^2)f''(x)+ (\lambda_1-2\beta-\lambda_1 x )f'(x)
\end{equation}
and $g(x)=\frac{x+1}{2}$, that
\begin{eqnarray*}
\bar{J}_{\beta} [f \circ g] (g^{-1}(x))&=&x(1-x)f''(x)+(\lambda_1 -\beta -\lambda_1 x)f'(x)
   = J_{\beta} [f](x).
\end{eqnarray*}
Then, the log-Sobolev constant being invariant by homeomorphism, one gets, from Saloff-Coste \cite{Saloff_Jacobi}, see also Fontenas \cite{fontenas:1998}, that the log-Sobolev constant $\alpha\left(\lambda_1,\beta\right)$ of the Jacobi operator $J_{\beta}$ is such that
\begin{equation}
\label{eq:log-Sobolev}
 \alpha\left(\lambda_1,\frac{\lambda_1}{2}\right)=\frac{\lambda_1}{2} \end{equation}
for the symmetric Jacobi and otherwise, $ \alpha\left(\lambda_1,\beta\right)<\frac{\lambda_1}{2}$  for $\lambda_1>2\beta$, with for any fixed $\beta$ and large $\lambda_1$, $\alpha\left(\lambda_1,\beta\right)\sim \frac{\lambda_1}{4} $.
Since always $\alpha\left(\lambda_1,\beta\right) \leq 2\lambda_1$,  we thus get, from \eqref{eq:log-Sobolev}, that the symmetric Jacobi semigroup attains the optimal entropic decay and hypercontractivity rate. We point out that the explicit expression of the log-Sobolev constant for the symmetric case goes back to Barky in \cite{MR1417973}. Although the log-Sobolev constant is not attainable in the other cases, the interweaving relation described above combined with theorems \ref{theo1} and \ref{theo2} enable us to provide the following information regarding the non-symmetric Jacobi semigroups.
\begin{pro}
 For any $\lambda_1>2\beta>1$, $m_0\in\cP((0,1))$ and $t\geq0$, we have
\bq
\Ent( {m}_0{\rm{J}}^{(\beta)}_{t+\tau^{(\lambda_1,\beta)}}\vert\nu_{\beta})&\leq &  e^{-\frac{\lambda_1}{2} t}  \: \Ent( {m}_0\vert\nu_{\beta})
 \eq
 and
 \begin{equation}
\label{eq:hyper}
\vvv {\rm{J}}^{(\beta)}_{t+\tau^{(\lambda_1,\beta)}}\vvv_{\esL^2(\nu)\ri\esL^{p(t)}} \leq 1  \textrm{ where }   p(t) =1+e^{ \frac{\lambda_1}{2}  t}.
\end{equation}
\end{pro}
We close this example by mentioning that in \cite{CPVS} interweaving relations are established between the symmetric Jacobi semigroup and a class of non-local and non-self-adjoint Markov semigroups on the unit interval $[0,1]$. 

\subsection{The non-self-adjoint generalized Laguerre semigroups} \label{jLp}
In this part, we illustrate that the concept of \cmir\ is also useful in the context of non-reversible and non-local Markov semigroups.
More specifically, let $P=(P_t)_{t\geq0}$ be the generalized Laguerre semigroup as introduced and thoroughly studied in \cite{Patie-Savov-GeL}. We also refer to this paper for further details about the objects that will be introduced in this part.
It can be characterized through its infinitesimal generator which takes the form, for a function $f$ smooth,
\begin{equation*}
{{L}}_{\phi}[f](x) =  xf''(x)+\lb m+1 -x\rb f'(x) + \int^{\infty}_{0} \left(f(e^{-y}x)-f(x)\right)\Pi(x,dy), \quad x>0,
\end{equation*}
where $m \geq 0$ and $\Pi(x,dy)=\frac{\Pi(dy)}{x}$ with $\Pi$  a finite non-negative Radon measure on $\R^+$ with a finite first moment, that is ${\rm{\overline{\Pi}}}=\int_{0}^{\infty}y\Pi(dy)<\infty$.
Observe that, writing $p_n(x)=x^n, x>0$, $n\in \N$, an integration by parts yields
\begin{equation*}
{{L}}_{\phi}[p_n](x) = n \phi(n) p_{n-1}(x) -n p_n(x),
\end{equation*}
where, for $u\geq0$, we have set
\begin{equation}\label{eq:def-b}
   \phi(u)=u+ m+  \int^{\infty}_{0} (e^{-uy}-1)\Pi(u,dy).
\end{equation}


Note that $\phi$ is a Bernstein function and it is in fact the Laplace exponent of the descending ladder height process $\xi=(\xi_t)_{t\geq 0}$ of the spectrally negative L\'evy process with Laplace exponent $u\phi(u)$, see e.g.~\cite[Sec.~6.5.2]{Kyprianou2014}.
$P$ admits  an unique invariant measure  which is an absolutely continuous probability measure with a density denoted by $\nu$.   Its law is determined by its integer moments which are given, for any $n\in \N$, by
 \[\int_0^{\infty}x^n \nu_{\phi}(x)dx=W_{\phi}(n+1)\]
 where $W_{\phi}(1)=1$ and $W_{\phi}(n+1)=\prod_{k=1}^{n} \phi(k)$.  $P$ extends to a non-self-adjoint strongly continuous contraction semigroup on ${\rm{L}}^{2}(\nu_{\phi})$.  Next, let $ \wi{P}^{(\beta)}=( \wi{P}^{(\beta)}_t)_{t\geq0}$ denotes  the semigroup of the classical Laguerre process of index $\beta\geq 0$ (or dimension $\beta +1$) and recall from Section \ref{cadLp} that its generator is the differential operator
\begin{equation*}
{L_{\beta+1}[f](x)} =   xf''(x)+\lb \beta +1 -x\rb f'(x), \quad x>0.
\end{equation*}
$ \wi{P}^{(\beta)}$ is a self-adjoint operator on ${\rm{L}}^{2}(\nu_{\beta})$ where here, for sake of simplicity, we write $\nu_{\beta}(dx)=\frac{x^{\beta}}{\Gamma(\beta+1)}e^{-x}dx,x>0$. We disregard  the parameter $\beta$ when it is $0$, that is we simply write $\wi P=\wi P^{(0)}$ and $\nu=\nu_{0}$.

Now, according to \cite{Patie-Savov-GeL}, there exists a multiplicative Markov kernel ${\rm{I}}_{\phi}$ defined by
\begin{equation}\label{eq:def-L}
  {\rm{I}}_{\phi} [f](x)=\E\left[f(xI_{\phi})\right], \quad x>0,
\end{equation}
where $I_{\phi}=\int_{0}^{\infty}e^{-\xi_t}dt$ with $\xi$ the subordinator with Laplace exponent the Bernstein function $\phi$ and,  for any $n\in \N$,
\begin{equation}\label{eq:L-pn}
  {\rm{I}}_{\phi} [p_n](x)=\frac{\Gamma(n+1)}{W_{\phi}(n+1)} p_n(x), \quad x>0.
\end{equation}
We also introduce for any $\beta>0$, the Markov kernel $ {\rm{B}}^*_{\beta}$, acting on any bounded Borelian function $f$ via
\begin{equation}\label{eq:V-pn}
  {\rm{B}}^*_{\beta} [f](x)= \frac{x^{\beta}}{\Gamma(\beta)}\int_0^{\infty}f((1+y)x)y^{\beta-1}e^{-yx}dy,\quad  x>0.
\end{equation}
We are ready to state and proof the following.
\begin{pro}\label{prop:cmirPQ}
For any $\beta>\overline{\Pi} +m$, we have \[\cont{P} \stackrel{\tau^{(\beta)}}{\leftrightsquigarrow}  \wi{P}^{(\beta)} \] where $\tau^{(\beta)}$  is an infinitely divisible variable characterized by
 \begin{equation}\label{eq:Lt-tau-beta}
    \int_{0}^{\infty}e^{-us}\P(\tau^{(\beta)} \in ds)=\left(\frac{\Gamma(1+\beta)\Gamma(u+1)}{\Gamma(u+\beta+1)}\right)=e^{-\phi_{\beta}(u)t}, \qquad u>0.
  \end{equation}
  In particular,  $\P(\tau^{(\beta)} \in ds)= (1+\beta)(1+\log s)^{\beta}ds, s\in (1/e,1)$.
  Moreover, for any such $\beta$,
we have 
\begin{equation}\label{eq:int-pqb}
 \Lambda={\rm{I}}_{\phi} {\rm{B}}^*_{\beta} \textrm{ and } \wi{\Lambda}={\rm{V}}_{\beta}
 \end{equation}
where ${\rm{V}}_{\beta}$ is a Markov kernel associated to the variable $Y_\beta$ whose distribution is determined  by its moments given by,  for any $n\in \N$,
 \begin{equation}\label{eq:V-pn}
  {\rm{V}}_{\beta}[p_n](x)=\E\left[p_n(xY_{\beta})\right]=\Gamma(1+\beta)\frac{W_{\phi}(n+1)}{\Gamma(n+1+\beta)} p_n(x), \quad x>0.
\end{equation}

 Finally, we have
  for any $t \geq 0$ and $m_0\in\cP((0,+\infty))$,
\begin{equation}
\label{eq:entr}
\Ent(m_0{P}_{t+\tau^{(\beta)}}\vert\nu_{\phi}) \leq e^{- t} \Ent(m_0\vert\nu_{\phi}),
\end{equation}
and
\begin{equation}
\label{eq:hyper}
\vvv P_{t+\tau^{(\beta)}}\vvv_{{\esL^2(\nu_{\phi})\ri\esL^{p(t)}(\nu_{\phi})}} \leq 1  \textrm{ where }   p(t)= 1+e^{ t}.
\end{equation}
  \end{pro}
\begin{proof}
First, we recall from \cite[Theorem 7.1]{Patie-Savov-GeL} that   the following intertwining relationship
\begin{equation}\label{eq:int1}
 P_t {\rm{I}}_{\phi}  = {\rm{I}}_{\phi}  \wi{P}_t, \quad  t\geq0,
 \end{equation}
holds  in ${\rm{L}}^{2}(\nu)$. 
Next, \cite[Proposition 4.4]{Patie-Savov-GeL} entails that, for any $\beta>
\overline{\Pi} +m$, $\phi_{\beta}(u)=\frac{\phi(u)}{u+\beta}$ is a Bernstein function and there exists a Markov kernel $V_\beta$ associated to the positive random variable $Y_{\beta}$ whose   moments are given by \eqref{eq:V-pn} and  determined its law.
Moreover, from Lemma 10.2 of the aforementioned paper, we have, in ${\rm{L}}^{2}(\nu_{\phi})$, the following  identity
\begin{equation}\label{eq:int2}
  \wi{P}^{(\beta)}_t  {\rm{V}}_{\beta}   =  {\rm{V}}_{\beta} P_t, \quad t\geq 0.
 \end{equation}
Then, invoking either \cite[Identity (1.c)]{MR1654531} or  again \cite[Theorem 7.1]{Patie-Savov-GeL}, we have  in ${\rm{L}}^{2}(\varepsilon)$
\begin{equation*}
    \wi{P}^{(\beta)}_t {\rm{B}}_{\beta}   =  {\rm{B}}_{\beta}  \wi{P}_t.
 \end{equation*}
Taking the adjoint, in the weighted Hilbert space, intertwining identity and using the fact that $\wi{P}$ (resp.~$\wi{P}^{(\beta)}_t$) is self-adjoint in ${\rm{L}}^{2}(\nu)$ (resp.~${\rm{L}}^{2}(\nu_{\beta})$) yields
  in ${\rm{L}}^{2}(\nu_{\beta})$
\begin{equation}\label{eq:int3}
    \wi{P}_t {\rm{B}}^*_{\beta}   =  {\rm{B}}^*_{\beta}  \wi{P}^{(\beta)}_t
 \end{equation}
Combining this with  \eqref{eq:int1} entails that in ${\rm{L}}^{2}(\nu_{\beta})$
\begin{equation*}
 P_t {\rm{I}}_{\phi} {\rm{B}}^*_{\beta}  = {\rm{I}}_{\phi}\wi{P}  {\rm{B}}^*_{\beta} = {\rm{I}}_{\phi} {\rm{B}}^*_{\beta}  \wi{P}^{(\beta)}_t.
 \end{equation*}
  Finally, this combines with the intertwining relationship \eqref{eq:int2} yields the identity in ${\rm{L}}^{2}(\nu_{\phi})$
\begin{equation}\label{eq:intlast}
 P_t {\rm{I}}_{\phi} {\rm{B}}^*_{\beta}   {\rm{V}}_{\beta}    = {\rm{I}}_{\phi_{\d}} {\rm{B}}^*_{\beta}  \wi{P}^{(\beta)}_t   {\rm{V}}_{\beta} = {\rm{I}}_{\phi_{d}} {\rm{B}}^*_{\beta} {\rm{V}}_{\beta}P_t
 \end{equation}
 and
\begin{equation}\label{eq:intlast}
  \wi{P}_t^{(\beta)} {\rm{V}}_{\beta}{\rm{I}}_{\phi} {\rm{B}}^*_{\beta} = {\rm{V}}_{\beta}{\rm{I}}_{\phi} {\rm{B}}^*_{\beta} \wi{P}_t^{(\beta)}.
 \end{equation}
 Since from \cite[Theorem 7.1(2) and Lemma 8.16]{Patie-Savov-GeL}, we have that ${\rm{I}}_{\phi}$ and ${\rm{B}}^*_{\beta}$ are one-to-one in ${\rm{L}}^{2}(\nu)$ and ${\rm{L}}^{2}(\nu_{\beta})$  respectively, we get that their composition  ${\rm{I}}_{\phi} {\rm{B}}^*_{\beta}$ is also one-to-one in ${\rm{L}}^{2}(\nu_{\beta})$.
 Thus,  it remains to show that $ P_{\tau}={\rm{I}}_{\phi} {\rm{B}}^*_{\beta}   {\rm{V}}_{\beta}$ or, by Theorem \ref{theo3b}, equivalently $\wi{P}_\tau^{(\beta)}={\rm{V}}_{\beta}{\rm{I}}_{\phi} {\rm{B}}^*_{\beta} $. To justify the latter identity, we proceed  as in the proof of Proposition \ref{complmonot}, we have from \cite[Theorem 1.22(c)]{Patie-Savov-GeL}, that, for any $t>0$, the spectrum of $P_t$ in ${\rm{L}}^{2}(\nu_{\phi})$ is discrete and given by $e^{-t\N}$ and each eigenvalue is simple with for all $n\in \N$,
\[P_t [\mathcal{P}_n](x)=e^{-nt}\mathcal{P}_n(x)\]
where the polynomials $\mathcal{P}_n$ are defined via the identity $\mathcal{P}_n(x)={\rm{I}}_{\phi} [\mathcal{L}_n](x)$,  $(\mathcal{L}_n)_{n\geq 0}$ being the orthonormal sequence of Laguerre polynomials. Thus, we deduce from \eqref{eq:intlast} that
\begin{equation*}
 P_t {\rm{I}}_{\phi} {\rm{B}}^*_{\beta}   {\rm{V}}_{\beta}[\mathcal{P}_n] =  {\rm{I}}_{\phi} {\rm{B}}^*_{\beta} {\rm{V}}_{\beta}P_t [\mathcal{P}_n] =e^{-nt}{\rm{I}}_{\phi} {\rm{B}}^*_{\beta} {\rm{V}}_{\beta}[\mathcal{P}_n],
 \end{equation*}
that is ${\rm{I}}_{\phi} {\rm{B}}^*_{\beta}   {\rm{V}}_{\beta}[\mathcal{P}_n]$ is proportional to $\mathcal{P}_n$.
 More specifically, recalling that for any $n\in \N$,
\[{\rm{B}}^*_{\beta} [p_n](x)=p_n(x)+ P_{n-1}(x) \]
where here and below $ P_{n-1}(x)$ stands for a generic polynomial of order $n-1$, we deduce  from \eqref{eq:L-pn} and \eqref{eq:V-pn} that
\begin{eqnarray}
  {\rm{V}}_{\beta} {\rm{I}}_{\phi} {\rm{B}}^*_{\beta}  [ p_n](x) &=& \frac{\Gamma(n+1+\d)}{\Gamma(1+\d)W_{\phi}(n+1)} \Gamma(1+\beta)\frac{W_{\phi}(n+1)}{\Gamma(n+1+\beta)} p_n(x)+ P_{n-1}(x) \\
   &=& \frac{\Gamma(1+\beta) \Gamma(n+1+\d) }{\Gamma(1+\d) \Gamma(n+1+\beta)} p_n(x)+ P_{n-1}(x)
\end{eqnarray}
and hence
\begin{eqnarray}
  {\rm{I}}_{\phi} {\rm{B}}^*_{\beta} {\rm{V}}_{\beta}  [\mathcal{P}_n] (x)
   &=& \frac{\Gamma(1+\beta) \Gamma(n+1+\d) }{\Gamma(1+\d) \Gamma(n+1+\beta)}  \mathcal{P}_n(x).
\end{eqnarray}
On the other hand, it is well known that  $\phi_\beta(u)=-\log\frac{\Gamma(1+\beta) \Gamma(u+1+\d) }{\Gamma(1+\d) \Gamma(n+1+\beta)} $ is a Bernstein function which corresponds to the Laplace exponent of the positive infinitely divisible variable  $\tau^{(\beta)}=-\log B_{\beta}$, where $B_{\beta}$ is  a beta variable of parameter $\beta>0$.
 Finally, since $\beta>\overline{\Pi} +m>0$,  one gets that the log-Sobolev constant of the classical Laguerre ${P}^{(\beta)}$ is $1$, see Remark \ref{rem:lsclag}. We complete the proof by invoking   theorems \ref{theo1} and \ref{theo2}.
\end{proof}

\subsubsection{Subordinate generalized Laguerre semigroups}
It is well-known, see e.g.~\cite{bakry:2014}, that, for any $t>0$,  $ \wi{P}^{(\beta)}_t$ is an Hilbert-Schmidt operator in ${\rm{L}}^{2}(\nu_{\beta})$ that  admits, for any $f \in  {\rm{L}}^{2}(\nu_{\beta})$, the diagonalization
\begin{equation}\label{eq:seK}
 \wi{P}^{(\beta)}_t [f] = \sum_{n=0}^{\infty}e^{-n t} \mathfrak{c}_{n}(\beta) \langle f, \mathcal{L}^{(\beta)}_n \rangle_{\nu_{\beta}}\:  \mathcal{L}^{(\beta)}_n
\end{equation}
where  the sequence of Laguerre polynomials $(\sqrt{\mathfrak{c}_{n}(\beta)}\mathcal{L}^{(\beta)}_n)_{n\geq 0}$ forms an orthonormal basis of ${\rm{L}}^{2}(\nu_{\beta})$ and we recall that
\begin{equation*}
\mathcal{L}^{(\beta)}_n(x)=\sum_{r=0}^n (-1)^r { n +\beta \choose n-r}  \frac{x^r }{r!}
\end{equation*}
and $\mathfrak{c}_{n}(\beta)=\frac{\Gamma(n+1) \Gamma(\beta+1)}{\Gamma(n+\beta+1)}$.
Moreover, a classical argument based on the spectral theory of reversible compact Markov semigoups yields, for any $t\geq 0$ and $f\in {\rm{L}}^{2}(\nu_{\beta})$, the spectral gap estimate
 \begin{equation}\label{eq:inv_Lag}
  \textrm{Var}_{\nu_{\beta}}\left( \wi{P}^{(\beta)}_t[f]\right)\leq e^{-t} \: \textrm{Var}_{\nu_{\beta}}\left(f\right)
 \end{equation}
where, we recall that for a measure $\nu$, we have set $\textrm{Var}_{\nu}\left(f\right)=||f-\nu [f]||^2_{\L^2(\nu)}$. Let us denote by $\wi{P}^{\tau^{(\beta)}}=(\wi{P}^{\tau^{(\beta)}}_t)_{t\geq 0}$  the Bochner subordination of $ \wi{P}^{(\beta)}$ by the subordinator $(\tau^{(\beta)}_t)_{t\geq 0}$ where $\tau^{(\beta)}_1$ has the same law than the positive infinitely divisible variable $\tau^{(\beta)}$ defined in Proposition \ref{prop:cmirPQ} and use the same notation for the subordinated semigroup ${P}^{\tau^{(\beta)}}$.

\begin{cor}
For any $\beta >0$, $t>0$,  $\wi P^{\tau^{(\beta)}}_t$ is a self-adjoint Hilbert-Schmidt operator in $\esL^2(\nu_{\beta})$ that  admits, for any $f \in  \esL^2(\nu_{\beta})$, the diagonalization
\begin{equation}\label{eq:seKwi}
\wi{P}^{\tau^{(\beta)}}_t [f] = \sum_{n=0}^{\infty} \mathfrak{c}^{t+1} _{n}(\beta) \: \langle f, \mathcal{L}^{(\beta)}_n \rangle_{\nu_{\beta}}\:  \mathcal{L}^{(\beta)}_n
\end{equation}
and
\begin{equation}\label{eq:inv_Lag}
  {\rm{Var}}_{\nu_{\beta}}\left( \wi{P}^{\tau^{(\beta)}}_t[f]\right)\leq (1+\beta)^{-t} \: {\rm{Var}}_{\nu_{\beta}}\left(f\right)
 \end{equation}
Moreover, for any $\beta> \overline{\Pi} +m$,  $\cont{P}^{\tau^{(\beta)}} \stackrel{1}{\leftrightsquigarrow}  \wi{P}^{\tau^{(\beta)}}$ and for any $f \in {\rm{L}}^{2}(\nu)$ and $t>1$, we have in  ${\rm{L}}^{2}(\nu)$
\begin{equation}\label{eq:spe}
  P^{\tau^{(\beta)}}_{t} [f]=\sum_{n=0}^{\infty}   \mathfrak{c}^{t} _{n}(\beta) \:  \langle f, \mathcal{V}_n \rangle_{\nu} \: \mathcal{P}_n
\end{equation}
and for any $t \geq 0$
\begin{equation}\label{eq:inv_GLag}
  {\rm{Var}}_{\nu}\left( {P}^{\tau^{(\beta)}}_t[f]\right)\leq (1+\beta)^{1-t}  \: {\rm{Var}}_{\nu}\left(f\right)
 \end{equation}
\end{cor}
\begin{proof}
The fact that  $\wi P^{\tau^{(\beta)}}$ is self-adjoint in $\esL^2(\nu_{\beta})$   can easily be checked  by means of Fubini theorem as, for any non-negative $f,g \in {\rm{L}}^{2}(\nu_{\beta})$ and $t\geq 0$,
\begin{eqnarray*}
  \langle \wi P^{\tau^{(\beta)}}_t[f], g \rangle_{\nu_{\beta}} &=& \int_{0}^{\infty} \langle \wi P^{{(\beta)}}_s [f], g \rangle_{\nu_{\beta}} \P(\tau^{(\beta)}_t \in ds) =  \int_{0}^{\infty} \langle f,  \wi P^{{(\beta)}}_s [g] \rangle_{\nu_{\beta}} \P(\tau^{(\beta)}_t \in ds)\\
  &=& \langle  f,  \wi P^{\tau^{(\beta)}}_t [g] \rangle_{\nu_{\beta}}
\end{eqnarray*}
where we used that $\wi P^{{(\beta)}}$ is self-adjoint in $\esL^2(\nu_{\beta})$.
Next, one has
that for any $f \in  \esL^2(\nu_{\beta})$, the diagonalization
\begin{eqnarray*}
\wi P^{\tau^{(\beta)}}_t [f] &=&  \int_{0}^{\infty}  \P(\tau^{(\beta)}_t \in ds) \wi P^{{(\beta)}}_s [f] \\
& = &   \int_{0}^{\infty} \P(\tau^{(\beta)}_t \in ds) \sum_{n=0}^{\infty}e^{-s n} \mathfrak{c}_{n}(\beta) \: \langle f, \mathcal{L}^{(\beta)}_n \rangle_{\nu_{\beta}}\:  \mathcal{L}^{(\beta)}_n \\
& = &    \sum_{n=0}^{\infty}\left(\frac{\Gamma(1+\beta)\Gamma(n+1)}{\Gamma(n+\beta+1)}\right)^{t} \mathfrak{c}_{n}(\beta) \: \langle f, \mathcal{L}^{(\beta)}_n \rangle_{\nu_{\beta}}\:  \mathcal{L}^{(\beta)}_n
\end{eqnarray*}
where  we used \eqref{eq:seK} in the second equality and to conclude we combined the identity \eqref{eq:Lt-tau-beta}, the Stirling formula that yields  that for  $n$ large enough
\begin{equation}\label{eq:stirl}
\mathfrak{c}_{n}(\beta)=\frac{\Gamma(n+1) \Gamma(\beta+1)}{\Gamma(n+\beta+1)} \sim \Gamma(\beta+1)  n^{-\beta}
\end{equation}
 with the fact that $ \wi{P}^{(\beta)}_t$ is closed as an Hilbert-Schmidt operator. Next, using the interweaving relation described in Proposition \ref{prop:cmirPQ} combined with Theorem \ref{theo3b} since $\tau^{(\beta)}$ is infinitely divisible, we get that for any $\beta> \overline{\Pi} +m$,  $\cont{P}^{\tau^{(\beta)}} \stackrel{1}{\leftrightsquigarrow}  \wi{P}^{\tau^{(\beta)}}$. From this relation,
 we deduce  that, for any $f \in  \esL^2(\nu)$ and $t>0$,
 \begin{eqnarray}
 P^{\tau^{(\beta)}}_{t+1}[f]    &=& P^{\tau^{(\beta)}}_{t} \Lambda_\beta   {\rm{V}}_{\beta}[f] =  \Lambda_\beta \wi P^{\tau^{(\beta)}}_{t}   {\rm{V}}_{\beta}[f] \\
    &=& \Lambda_\beta \sum_{n=0}^{\infty}\left(\frac{\Gamma(1+\beta)\Gamma(n+1)}{\Gamma(n+\beta+1)}\right)^{t} \mathfrak{c}_{n}(\beta) \: \langle {\rm{V}}_{\beta} [f], \mathcal{L}^{(\beta)}_n \rangle_{\nu_{\beta}}\:  \mathcal{L}^{(\beta)}_n\\
   &=& \sum_{n=0}^{\infty}  \mathfrak{c}^{t} _{n}(\beta) \mathfrak{c}_{n}(\beta) \:  \langle f, \mathcal{V}_n \rangle_{\nu} \: \mathcal{P}_n
 \end{eqnarray}
 where $\mathcal{V}_n = {\rm{V}}^*_{\beta} \mathcal{L}^{(\beta)}_n$  and $\Lambda_\beta \mathcal{L}^{(\beta)}_n = {\rm{I}}_{\phi} {\rm{B}}^*_{\beta}  \mathcal{L}^{(\beta)}_n={\rm{I}}_{\phi} \mathcal{L}_n=\mathcal{P}_n(x)$, which completes the proof of the spectral expansion
 of $P^{\tau^{(\beta)}}_{t}f$ for $t>1$. The last claim follows from the interweaving relation with warm-up time $1$ and an application of Theorem \ref{theo10} below by choosing $\varphi(x)=x^2-1$   \end{proof}

\section{Proofs of the main results} 
\label{poTaP}

In the following subsections, we prove the main results about \cmirs\  announced in the introduction.

\subsection{Proof of the results from section \ref{sec:basic_prop}}
\subsubsection{Proof of Theorem \ref{theo3b}}
Here we consider warm-up distributions which are infinitely divisible  distributions and we construct via subordination other  interweaved Markov semigroups  which  brought us back to the situation of deterministic warm-up times, thus showing Theorem \ref{theo3b}.  More precisely, assume that $\boldsymbol{\wu}$ is infinitely divisible. Then there exists a unique convolution semigroup on $\R^+$ which determines  the transition kernel of the subordinator $(\wu_t)_{t\geq 0}$ where $\wu_1 \stackrel{(d)}{=}\boldsymbol{\wu}$.
Given a Markov semigroup $P$,
define the family of Markov operators $Q\df(Q_t)_{t\geq 0}$ via
\bqn{Q}
\nonumber\fo t\geq 0,\qquad Q_t&\df&P_{\wu_t}
=\int_{\RR_+} P_s\, \wu_t(ds)\eqn
$Q$ is the subordination of $P$  in the sense of Bochner and it is also a Markov semigroup, see e.g.~\cite[Chap.~12]{schilling:2012}. 
Similarly, given another Markov semigroup $\wi P$,
define the  Markov  semigroup $\wi Q\df(\wi Q_t)_{t\geq 0}\df (\wi P_{\wu_t})_{t\geq 0}$.
\par
As in Theorem \ref{theo3b}, assume an \cmir\  holds between the semigroups $P$ and $\wi P$ with warm-up distribution $\boldsymbol{\wu}$, that is  $\cont{P} \stackrel{\boldsymbol{\wu}}{\looparrowleft} \wi{P}$.
Denote by $\Lambda$ and $\wi\Lambda$ the corresponding Markov kernels between the underlying state spaces $V$ and $\wi V$
Then Figure \ref{fig1} leads to the following diagram for all $t\geq 0$.
 \begin{figure}[H]\centering
\begin{tikzcd}[scale=2]
V\arrow[r, "Q_t"] \arrow[d, "\Lambda"'] \arrow[dd, bend right=50 , "Q_1"']
& V\arrow[d, "\Lambda" ]  \arrow[dd, bend left=50 , "Q_1"]\\
\wi V\arrow[r, "\wi Q_t" ]\arrow[d, "\wi\Lambda"']
&\wi V\arrow[d, "\wi\Lambda" ]\\
V\arrow[r, "Q_t"]
& V
\end{tikzcd}
\caption{Intertwining relations for $Q$ and $\wi Q$}\label{fig7}
\end{figure}
\par
Indeed, by definition, we have $\Lambda\wi\Lambda =P_{\boldsymbol{\wu}}=Q_1$ and for any $t\geq 0$, we get
\bq
Q_t\Lambda&=&\int_{\RR_+} P_s\, \wu_t(ds)\Lambda\\
&=&\int_{\RR_+} P_s\Lambda\, \wu_t(ds)\\
&=&\int_{\RR_+} \Lambda \wi P_s\, \wu_t(ds)\\
&=&\Lambda \wi Q_t\eq
Similarly, we have
\bq
\fo t\geq 0,\qquad \wi Q_t\wi \Lambda&=&\wi\Lambda Q_t\eq
and this ends the proof of Theorem \ref{theo3b}.
\par



\subsubsection{Proof of Theorem \ref{thm:ref-cmir}}\label{ecte}
  The first claim is obvious. Next,  if $\cont{P} \stackrel{\tau}{\looparrowleft} \wi{P}   $ with $ \Cmir{\cont{P}}{{\Lambda}}{\wi{P}  }{\wi{\Lambda}}$, then, clearly $\Cmir{\wi{P}  }{\wi{\Lambda}}{\cont{P}}{{\Lambda}}$. Moreover, since  Markovian intertwining relationship is stable by mixture with a positive measure, we get that
    that $\Gate{\cont{P^\tau}}{{\Lambda}}{\cont{Q^\tau}}$ and as $P^{\tau}={\Lambda \wi{\Lambda}}$, we get
     \[{\Lambda} \left({\rm VI}-Q^\tau\right)=0\]
     which  concludes the proof of \eqref{it:T1} by an injectivity argument. Next, if  $ \Cmir{\cont{P}}{{\Lambda}}{\wi{P}  }{\wi{\Lambda}}$ and  $\Cmir{\wi{P}  }{\overline{\rm V}}{\overline{\cont{P}}}{\overline{\Lambda}}$ then $\Cmir{\cont{P}}{{\Lambda}\overline{\rm V}}{\overline{\cont{P}}}{\overline{\Lambda}\wi{\Lambda}}$. Moreover,   we have
      \[ {\Lambda}\overline{\rm V}  \: \overline{\Lambda}\wi{\Lambda}={\Lambda}Q^{\overline{\tau}}\wi{\Lambda}={\Lambda}\wi{\Lambda}{P}^{\overline{\tau}}=P^{\tau}{P}^{\overline{\tau}}= F(L)\overline{F}(L) \]
      where we used successively that $\wi{P}   \stackrel{\overline{\tau}}{\looparrowleft} \overline{P} $, $ \Gate{\wi{P}  ^{\tau}}{\wi{\Lambda}}{\cont{P}^{\tau}}$ which itself follows as above from $ \Gate{\wi{P}  }{\wi{\Lambda}}{\cont{P}}$,  $\cont{P} \stackrel{\tau}{\looparrowleft} \wi{P}   $  and the last identity sets a notation.  To complete the proof we observe that the product $F\overline{F}$ is the Laplace transform of the sum of the independent random variables $\tau+\overline{\tau}$.


\subsubsection{Proof of Theorem \ref{thm:transp}}
 First, by  since $ \Cmir{\cont{P}}{{\Lambda}}{\wi{P}  }{\wi{\Lambda}}$ and   $\Gate{{P}}{V}{{P^V}}$ and  $\Gate{{P^V}}{V^{-1}}{{P}}$, we easily deduce that  $ \Cmir{\cont{P}^V}{{V \Lambda}}{\wi{P}  }{\wi{\Lambda}V^{-1}}$ and we conclude the proof of the first item by observing that  $V \Lambda \wi{\Lambda}V^{-1}=  V\cont{P}_{\tau} V^{-1}=\cont{P}^V_{\tau}$.
Next,  the identities \eqref{eq:intpq}, \eqref{eq:iLi} and the second gateway in \eqref{eq:cmir}  yield
      \[\Lambda  \discret{P} \discret{I} = P \Lambda \discret{I}  =  P {\rm{I}} \Lambda = {\rm{I}} Q \Lambda =  {\rm{I}} \Lambda  \wi{\discret{P}}  =  \Lambda \discret{I} \wi{\discret{P}}\]
    and the injectivity of $\Lambda$ gives that $\Gate{\discret{P}}{\discret{I}}{\wi{\discret{P}}}$. On can interchange the role of $\discret{P}$ and $\wi{\discret{P}}$ in the previous sequence of identities to conclude that $\Cmir{\discret{P}}{{\discret{I}}}{\wi{\discret{P}}}{{\discret{V}}}$.
    Next, as above, by  stability of intertwining relation by mixture, we get
    that $\Gate{\cont{P^\tau}}{\Lambda}{\discret{P^\tau}}$ and hence $\Lambda \discret{P}^{\tau}=\cont{P}^\tau\Lambda= {\Lambda \wi{\Lambda}} \Lambda = \Lambda {\discret{IV}}$ which concludes the proof by invoking the injectivity of $\Lambda$.

\subsection{Extensions and proofs  of the results from Section \ref{sec:appl}}
\subsubsection{Proof of Theorem \ref{theo1}}\label{ecte}

Here we extend the statement of Theorem  \ref{theo1} by considering (relative) $\varphi$-entropies.\par
Let $\varphi\st \RR_+\ri\RR_+$ be a convex function such that $\varphi(1)=0$.
The (relative) $\varphi$\textbf{-entropy} of two probability measures $m$ and $\nu$ defined on the same state space is given by
\bq
\Ent_{\varphi}(m\vert\nu)&\df&
 \int \varphi\lt(\frac{dm}{d\nu}\rt)\, d\nu+\lt(1-\int \frac{dm}{d\nu}\, d\nu\rt)\lim_{x\ri+\iy}\frac{\varphi(x)}{x}
\eq
where $dm/d\nu$ stands for the Radon-Nikodym density of $m$ with respect to $\nu$.
In this definition the convention $0\cdot\iy=0$ is enforced, namely, when $m$ is absolutely continuous with respect to $\nu$, the second term vanishes.
When  $m$ is not absolutely continuous with respect to $\nu$, i.e.\ $\int \frac{dm}{d\nu}\, d\nu<1$, their $\varphi$\textbf{-entropy} is $+\iy$ as soon as $\lim_{x\ri+\iy}\frac{\varphi(x)}{x}=+\iy$.
The case of the usual entropy $\Ent_\varphi(\cdot)$ corresponds to the particular function $\varphi$ given by
\bqn{Entcla}
\fo x\in\RR_+,\qquad \varphi(x)&\df& x\ln(x)-x+1\eqn
\par\sm
Recall the framework of the introduction: $P$ and $\wi P$ are two Markov semigroups, respectively on the state spaces $V$ and $\wi V$.
Let $\Lambda$ and $\wi \Lambda$ be Markov kernels from $V$ to $\wi V$ and from $\wi V$ to $V$. We assume that $P$ and $\wi P$ admit invariant probability measures $\nu$ and $\wi \nu$
and that $\nu\Lambda=\wi\nu$ and $\wi\nu\wi\Lambda=\nu$. Estimates in the $\varphi$-entropy sense on the speed of convergence to equilibrium for $\wi P$
can be transferred to $P$ with the help of a c.m.i.r.:\par
\begin{theo}\label{theo10}
Assume that there exists a interweaving relation   from $P$ to $\wi P$ with warm-up distribution $\wu$ and that
\bqn{vareps2}
\fo \wi m_0\in\cP(\wi V),\,\fo t\geq 0,\qquad \Ent_\varphi(\wi m_0\wi P_t\vert\wi\nu)&\leq &\varepsilon(t, \Ent_\varphi(\wi m_0\vert\wi\nu))\eqn
for some function $\varepsilon\st \RR_+\times \overline{\RR}_+\ri \overline{\RR}_+$, which is
 non-decreasing with respect to the second variable.
Then we have
\bqn{espE2}
\fo  m_0\in\cP(V),\,\fo t\geq 0,\qquad \Ent_\varphi( m_0P_{\theta_t(\wu)}\vert\nu)&\leq &\varepsilon(t, \Ent_\varphi(m_0\vert\nu))\eqn
where $\theta_t$ is the translation operator on $\RR_+$.
\end{theo}
\begin{remark}
As in the introduction, for this estimate to be meaningful, one should furthermore require that
\bq
\fo E\in\RR_+,\qquad \lim_{t\ri +\iy} \varepsilon(t,E)&=&0\eq
\par
\end{remark}
\prooff{Proof of Theorem \ref{theo10}}
Consider $E$ and $\wi E$ two measurable spaces and $\Xi$ a Markov kernel from $E$ to $\wi E$.
Let $\wi m$ and $m$ be two probability measures on $E$. As a consequence of Jensen inequality, we have for any convex function $\varphi$ as above,
\bqn{Xi}
\Ent_{\varphi}(\wi m\Xi\vert m\Xi)&\leq & \Ent_{\varphi}(\wi m\vert m)
\eqn
(see e.g.\ \cite{MR1992499}).\par
The interweaving relation  between $P$ and $\wi P$ implies that for any $t\geq 0$, we have
\bq
\Lambda \wi P_t\wi\Lambda &=&P_{\theta_t(\wu)}\eq
It follows that  for any $m_0\in\cP(V)$,
\bq
\Ent_\varphi( m_0P_{\theta_t(\wu)}\vert\nu)&=&\Ent_\varphi( m_0\Lambda \wi P_t\wi\Lambda\vert\wi\nu\wi\Lambda)\\
&\leq & \Ent_\varphi( m_0\Lambda \wi P_t\vert\wi\nu)\eq
where \eqref{Xi} was applied with $\wi m\df m_0\Lambda \wi P_t$, $m\df\wi\nu$ and $\Xi\df \wi\Lambda$.
Taking into account \eqref{vareps2}, we get
\bq
\Ent_\varphi( m_0\Lambda \wi P_t\vert\wi\nu)&\leq & \varepsilon(t,\Ent_\varphi(m_0\Lambda\vert \wi\nu))\\
&=&\varepsilon(t,\Ent_\varphi(m_0\Lambda\vert\nu\Lambda))\\
&\leq & \varepsilon(t,\Ent_\varphi(m_0\vert\nu))
\eq
where we used again \eqref{Xi} with $\wi m\df m_0$, $m\df\nu$ and $\Xi\df \Lambda$.
\wwtbp\par
The traditional way to deduce a bound such as \eqref{vareps2} is via $\varphi$-Sobolev inequalities.
Without entering into the general theory, let us e.g.\ consider the case where $\wi V$ is a finite state space and $\wi P$ is generated by
an irreducible Markov generator $\wi L$. Denote $\wi\cA$ the set of positive functions defined on $\wi V$ with $\wi\nu[f]=1$ and assume that $\varphi$ is differentiable on $(0,+\iy)$
(in particular $\varphi'(1)=0$).
Consider the energy
\bq
\fo f\in\wi\cA,\qquad \wi\cE_\varphi(f,\varphi'(f))&\df&- \wi\nu[f\wi L[\varphi'(f)]]\eq
(the r.h.s.\ is always non-negative) and denote
\bq
\wi\alpha_\varphi&\df& \inf_{f\in\wi\cA\setminus\{\wi\un\}}\frac{\wi\cE_\varphi(f,\varphi'(f))}{\Ent_\varphi(f\cdot\wi\nu\vert\wi\nu)}\eq
where $\wi\un$ is the function only taking the value 1 on $\wi V$ and $f\cdot\wi\nu$ is the probability on $\wi V$ admitting the density $f$ w.r.t.\ $\wi\nu$.
The quantity $\wi\alpha_\varphi$ is non-negative and is called the $\varphi$-Sobolev constant.
Then \eqref{vareps2} holds with the function $\varepsilon$ given by
\bq
\fo t\geq 0,\,\fo E\geq 0,\qquad \varepsilon(t,E)&\df& \exp(-\wi\alpha_\varphi t)E\eq
\par
This result is obtained by differentiating the quantity $\Ent_\varphi(\wi m_0\wi P_t\vert\wi\nu)$ with respect to $t>0$, for any fixed $\wi m_0\in\cP(\wi V)$,
and by applying Gr\" onwall lemma. The validity of this approach is very general, up to the appropriate definition of the domain $\wi\cA$.
\par
In the classical case \eqref{Entcla} and when the finite generator $\wi L$ is assumed to be furthermore reversible, the energy is given by
\bq
\fo f\in\wi\cA,\qquad \wi\cE(f,\ln(f))&=&\frac12\sum_{x,y\in \wi V} (f(y)-f(x))(\ln(f(y))-\ln(f(x)))\, \wi\nu(x)\wi L(x,y)
\eq
and the corresponding constant $\wi \alpha$ is called the modified logarithmic Sobolev constant.
It is bounded below by the usual logarithmic Sobolev constant, obtained by replacing  $\wi\cE(f,\ln(f))$ by
\bq
4\wi\cE(\sqrt{f},\sqrt{f})&=&2\sum_{x,y\in \wi V} (\sqrt{f}(y)-\sqrt{f}(x))^2\, \wi\nu(x)\wi L(x,y)
\eq
in the above definitions. In the diffusion framework,  the modified and usual logarithmic Sobolev constant coincide (for the previous functional analysis assertions, see for instance the book of Ané et al.\ \cite{MR2002g:46132}).
\par\me
Let us consider the situation of a deterministic warm-up time: there exists $t_0\geq 0$ such that $\wu=\delta_{t_0}$, as in Section \ref{dwte}.
Assume that $(\wi P,\wi\nu)$ satisfies a modified logarithmic Sobolev inequality with constant $\wi\alpha>0$, so that
for any initial distribution $\wi m_0\in\cP(\wi V)$, we have
\bq
\fo t\geq 0,\qquad \Ent(\wi m_{t}\vert\wi \nu)&\leq & \exp(-\wi\alpha t)\Ent(\wi m_0\vert\wi\nu)\eq
\par
Theorem \ref{theo10} enables to get for $(P,\nu)$ that
for any initial distribution $m_0\in\cP( V)$, we have
\bq
\fo t\geq 0,\qquad
\Ent_{\varphi}(m_{t_0+t}\vert \nu)&\leq & \exp(-\wi\alpha t)\Ent_{\varphi}(m_0\vert\nu)\eq
Alternatively, taking into account that the relative entropy of the time marginal laws of a Markov process with respect to its invariant measure is always non-increasing
with respect to time (see e.g.\ \cite{MR1992499}), we get
\bqn{Villain}
\fo t\geq 0,\qquad
\Ent_{\varphi}(m_{t}\vert \nu)&\leq &  \exp(-\wi\alpha (t-t_0)_+)\Ent_{\varphi}(m_0\vert\nu)\eqn
In this bound, the time $t_0$  clearly appears as a warm-up period. The fact that no contractive estimate of $\Ent_{\varphi}(m_{t}\vert \nu)$ can be deduced for $t\in[0,t_0]$
relates \eqref{Villain} to hypocoercive bounds (see e.g.\ Villani \cite{MR2562709}).\par
These considerations were illustrated by the classical and discrete examples of Subsection \ref{cadLp}. In Subsection \ref{jLp}, we presented a interweaving relation  with a random warm-up time
between jump Laguerre processes and classical Laguerre processes. It enables to get estimates on convergence to equilibrium in entropy sense for non-reversible jump
processes without the a priori knowledge of corresponding modified logarithmic Sobolev inequalities. It shows the  applicative potential of c.m.i.r.
 \begin{rem}
In general, it is not possible to deduce from a bound such as \eqref{espE2} an estimate on
$\Ent_\varphi( m_0P_{t}\vert\nu)$ for given large $t\geq 0$, except in the case of a deterministic warm-up time.
Indeed, consider $P$ the deterministic semigroup generated on the circle $\TT\df\RR/(2\wu\ZZ)$ by the usual derivation $\pa$.
Starting from $x_0\in\TT$, the position at time $t\geq 0$ of an associated Markov process is $x_0+t\ [2\wu]$. The associated invariant measure $\nu$ is the uniform distribution over $\TT$.
Let $\wu$ be the uniform distribution over $[0,2\wu]$.
For any $t\geq 0$, we have $\Ent_\varphi( m_0P_{\theta_t(\wu)}\vert\nu)=0$ for any initial distribution $m_0$, while $\Ent_\varphi( m_0P_{t}\vert\nu)=+\iy$ when $m_0$ is a Dirac mass.
\end{rem}
\par
\begin{rem}
Another approach to convergence to equilibrium is based on strong stationary times, see Aldous  and Diaconis \cite{MR876954} and Diaconis and Fill \cite{MR1071805} for seminal works about this alternative point of view.
It is more probabilistic in  spirit, since it constructs stopping times $\tau$ such that the position of the underlying Markov process is at equilibrium  and independent from $\tau$.
Furthermore, it is an important motivation for the investigation of intertwining relations.
Thus it is natural to wonder if interweaving relations  enable the transfer of strong stationary times. Unfortunately we did not find a satisfactory procedure, especially when the warm-up distribution is not a Dirac mass.
Nevertheless,
strong stationary times are often used due to their close relation to the convergence to equilibrium in the separation sense (see e.g.\ Diaconis and Fill \cite{MR1071805}),
and \cmirs\  enable to directly transfer corresponding estimates.\par\sm
Recall that the separation discrepancy $\fs(m,\nu)$ between two probability measures $m$ and $\nu$ on the same state space is defined as
\bq
\fs(m,\nu)&\df& \esssup_{\nu} 1-\frac{dm}{d\nu}\eq
\par
The separation discrepancy is in fact a limit case of $\varphi$-entropies. More precisely, for $p\geq 1$ , consider the convex mapping
\bq
\fo x\in\RR_+,\qquad \varphi_p(x)&\df& (1-x)_+^p\eq
where $(\cdot)_+$ stands for the non-negative part. It is not difficult to show that
for any probability measures $m$ and $\nu$ on the same state space, we have
\bq
\lim_{p\ri+\iy} \lt(\Ent_{\varphi_p}(m,\nu)\rt)^{1/p}&=&\fs(m,\nu)\eq\par
This result in conjunction with Theorem \ref{theo1} show that we can transfer separation estimates through c.m.i.r.
More precisely, assume that we have a  \cmir\  with warm-up distribution $\wu$ between the ergodic semigroups $P$ and $\wi P$, with invariant probability $\nu$ and $\wi\nu$.
Let $m_0$ be an initial distribution on $V$ and denote $\wi m_0\df m_0\Lambda$.
Assume that we have a function $\wi\varepsilon\st\RR_+\ri\RR_+$ such that
\bq
\fo t\geq 0,\qquad \fs(\wi m_0\wi P_t, \wi\nu)&\leq & \wi\varepsilon(t)\eq
Since we have for any $p\geq 1$ and any probability measure $\wi m$ on $\wi V$,
\bq
\Ent_{\varphi_p}(\wi m,\wi \nu)&\leq & \fs(\wi m,\wi\nu)^p\eq
Theorem \ref{theo1} implies that
\bq
\fo p\geq 1,\,\fo t\geq 0,\qquad \Ent_{\varphi_p}(m_0P_{\theta_t(\wu)},\nu)&\leq & \wi\varepsilon(t)^p\eq
It remains to take the power $1/p$ and to let $p$ go to infinity to get
\bq
\fo t\geq 0,\qquad \fs(m_0P_{\theta_t(\wu)},\nu)&\leq & \wi\varepsilon(t)\eq
which corresponds to the wanted separation estimate transfer.
\end{rem}

\subsection{Hyperboundedness}

As in the previous subsection, the underlying principle for the transfer of hyperboundedness via \cmirs\  is convexity, so that the Orlicz spaces are the natural framework here,
not only the $\esL^p$ spaces, for $p\geq 2$, as stated in Theorem \ref{theo2}.\par\sm
Let us recall the notion of Orlicz spaces (for a general introduction, see for instance the book of Rao and Ren \cite{MR1113700}).
Let $\varphi\st\RR\ri\RR_+$ be a Young function: it is a even convex function $\varphi\not=0$ satisfying $\varphi(0)=0$.
When $E$ is a measurable space endowed with a probability measure $m$, the \textbf{Orlicz space} $\esL^{\varphi}(m)$ is the vector space
of measurable functions $f\st E\ri\RR$ such that
\bq
\lVe f\rVe_{ \esL^{\varphi}(m)}&\df& \inf\{r>0\st \int \varphi(f/r)\, dm\leq 1\}\eq
is finite. The quantity $\lVe \cdot\rVe_{ \esL^{\varphi}(m)}$ defines a norm on $\esL^{\varphi}(m)$, when the functions are identified up to a $m$-negligible set.
The key property of Orlicz spaces we will need is:
\begin{lem}\label{OrliczLambda}
Consider $\Lambda$ a Markov kernel from $E$ to another measurable space $\wi E$.
Let $\wi m$ be the image of the probability measure $m$ on $E$ by $\Lambda$.
For any measurable function $f\st \wi E\ri\RR$, we have
\bq
\lVe \Lambda[f]\rVe_{\esL^{\varphi}(m)}&\leq & \lVe f\rVe_{\esL^{\varphi}(\wi m)}
\eq
\end{lem}
\proof
This is an immediate consequence of convexity.
Indeed, by
Jensen's inequality, we have $m$-a.s.\ and for any $r\geq 0$,
\bq
\varphi(\Lambda[f/r])&\leq & \Lambda[\varphi(f/r)]\eq
Integrating with respect to $m$, we get
\bq
\int \varphi(\Lambda[f/r])\, dm&\leq & \int \Lambda[\varphi(f/r)]\, dm\\
&=& \int \varphi(f/r)\, d\wi m
 \eq
and it remains to take the infimum of the $r>0$ such that
$\int \varphi(f/r)\, d\wi m\leq 1$ to get the announced result.
\wwtbp
\par
As in the introduction,
let be given $P$ a Markov semigroup from $V$ to $V$ and $\wi P$  a Markov semigroup from $\wi V$ to $\wi V$.
Assume that $\nu$ and $\wi \nu$ are respectively invariant probability measures for $P$ and $\wi P$ and that an \cmir\  holds,
as described in Figure \ref{fig1}, with Markov kernels $\Lambda$ from $V$ to $\wi V$ and $\wi\Lambda$ from $\wi V$ to $V$,
as well as warm-up distribution $\wu$. As usual, $\nu\Lambda$ and $\wi\nu\wi\Lambda$ are respectively invariant for $\wi P$ and $P$.
In case of non-uniqueness of these invariant probability measures, we furthermore assume that $\wi\nu=\nu\Lambda$ and $\nu=\wi\nu\wi\Lambda$.
Here is an extension of Theorem \ref{theo2}:
\begin{theo}\label{Orliczcmir}
Assume that for some time $T\geq 0$ and some Young function $\varphi$, we have in the operator norm
\bqn{hypercon2}
\vvv \wi P_T\vvv_{\esL^2(\wi\nu)\ri\esL^{\varphi}(\wi\nu)}&\leq &1\eqn
Then we get
\bqn{eqtheoO}
  \vvv P_{T+\wu}\vvv_{\esL^2(\nu)\ri\esL^{\varphi}(\nu)}&\leq & 1\eqn
\end{theo}
\proof
As in the proof of Theorem \ref{theo10}, the starting point is
\bq
\Lambda \wi P_T\wi\Lambda &=&P_{T+\wu}\eq
It follows that
\bq
 \vvv P_{T+\wu}\vvv_{\esL^2(\nu)\ri\esL^{\varphi}(\nu)}&\leq &
 \vvv\Lambda\vvv_{\esL^{\varphi}(\wi \nu)\ri \esL^{\varphi}(\nu)}
 \vvv \wi P_T\vvv_{\esL^2(\wi\nu)\ri\esL^{\varphi}(\wi\nu)}
 \vvv\wi\Lambda\vvv_{\esL^2(\nu)\ri\esL^2(\wi\nu)}\\
 &\leq &
 \vvv\Lambda\vvv_{\esL^{\varphi}(\wi \nu)\ri \esL^{\varphi}(\nu)}
 \vvv\wi\Lambda\vvv_{\esL^2(\nu)\ri\esL^2(\wi\nu)}\\
\eq
\par
Lemma \ref{OrliczLambda} applied with $m=\nu$ and $\wi m=\wi\nu$ (recall that $\nu\Lambda=\wi\nu$) implies
that
\bq  \vvv\Lambda\vvv_{\esL^{\varphi}(\wi \nu)\ri \esL^{\varphi}(\nu)} &=&1\eq
Considering the Young function $\RR\ni x\mapsto x^2$, Lemma \ref{OrliczLambda} applied with $m=\wi\nu$ and $m=\nu$ (recall that $\wi\nu\wi\Lambda=\nu$) implies
that
\bq \vvv\wi\Lambda\vvv_{\esL^2(\nu)\ri\esL^2(\wi\nu)}&=&1\eq
concluding the proof of the wanted bound.
\wwtbp
\par
Theorem \ref{theo2} is a consequence of Theorem \ref{theo10}, applied, for fixed $t\geq 0$,  with $T=t$ and
\bq
\varphi  \st \RR\ni x&\mapsto & x^{p(\wi\alpha t)}\eq
Note that due to the warm-up distribution, it is not possible to deduce from the conclusion of Theorem \ref{theo2}
that the semigroup $P$ satisfies a logarithmic Sobolev inequality (for the classical links between the latter inequality and hypercontractivity,
again see e.g.\ Ané et al.\ \cite{MR2002g:46132}).

\subsection{Proof of Theorem \ref{theo3}}

Assume first  that a cut-off phenomenon occurs for the family $(P^{(n)})_{n\in\ZZ_+}$, with cut-off times $(t^{(n)})_{n\in\ZZ_+}$,
and let us show the same is true for
$(\wi P^{(n)})_{n\in\ZZ_+}$.\par\sm
Consider the Young function $\RR\ni x\mapsto \lve x-1\rve$. The associated entropy between the probability measures $m$ and $\nu$ is just twice the total variation
\bq
2\lVe m-\nu\rVe_{\mathrm{tv}}&=&\int \lve \frac{dm}{d\nu}-1\rve\, d\nu+1-\int \frac{dm}{d\nu}\, d\nu \eq
The proof of
 Theorem \ref{theo10} with this particular Young function shows that for any $n\in\ZZ_+$,
 \bq
 \fo t\geq 0,\,\fo \wi m_0\in\cP(\wi V^{(n)}),\qquad \lVe \wi m_0\wi P^{(n)}_{t_0^{(n)}+t}-\wi\nu^{(n)}\rVe_{\mathrm{tv}}&\leq &  \lVe \wi m_0\wi\Lambda P^{(n)}_{t}-\nu^{(n)}\rVe_{\mathrm{tv}}\\
 &\leq & \frd^{(n)}(t)
\eq
where $\frd^{(n)}$ is given in \eqref{frdn}. Considering a similar definition of $\wi\frd^{(n)}$ for the semigroup $\wi P^{(n)}$, we obtain
\bq
\fo t\geq 0,\qquad \wi\frd^{(n)}(t_0^{(n)}+t)&\leq & \frd^{(n)}(t)\eq
Taking into account that for any $n\in\ZZ_+$, the function $\wi\frd^{(n)}$ is non-increasing, we deduce from the cut-off phenomenon for $P$ and from \eqref{neglig} that for any $r>0$,
\bq
\limsup_{n\ri\iy}\wi\frd^{(n)}((1+r)t^{(n)})
&\leq &
\limsup_{n\ri\iy}\wi\frd^{(n)}(t_0^{(n)}+(1+r/2)t^{(n)})\\
&\leq &
\lim_{n\ri\iy}\frd^{(n)}((1+r/2)t^{(n)})\\
&=&0\eq
\par
For the other point in the definition of the cut-off phenomenon, assume by contradiction that for some $r_0\in (0,1)$, we have
\bqn{contrad}
\liminf_{n\ri\iy}\wi\frd^{(n)}((1-r_0)t^{(n)})&<&1\eqn
By the assumed symmetry of the \cmirs\  between the sequence $(P^{(n)})_{n\in\ZZ_+}$ and $(\wi P^{(n)})_{n\in\ZZ_+}$,
we show as above that
\bq
\fo t\geq 0,\qquad \frd^{(n)}(t_0^{(n)}+t)&\leq & \wi\frd^{(n)}(t)\eq
Taking into account that for any $n\in\ZZ_+$, the function $\frd^{(n)}$ is non-increasing, we deduce from \eqref{contrad} and from \eqref{neglig} that
\bq
\lim_{n\ri\iy}\frd^{(n)}((1-r_0/2)t^{(n)})&\leq &
\liminf_{n\ri\iy}\frd^{(n)}(t_0^{(n)}+(1-r_0)t^{(n)})\\
&\leq &
\liminf_{n\ri\iy}\wi\frd^{(n)}((1-r_0)t^{(n)})\\
&<&1\eq
which is in contradiction with the cut-off phenomenon  for the family $(P^{(n)})_{n\in\ZZ_+}$.
Thus we get that for any $r\in(0,1)$,
\bq
\lim_{n\ri\iy}\wi\frd^{(n)}((1-r)t^{(n)})&=&1\eq
and this ends the proof that a cut-off phenomenon occurs for the family $(P^{(n)})_{n\in\ZZ_+}$
with cut-off times $(t^{(n)})_{n\in\ZZ_+}$.\par\sm
The remaining claims of Theorem \ref{theo3} are proven by a similar line of reasoning.

\vskip2cm
\hskip50mm
\vbox{
\copy5
\vskip5mm
\copy6
}

\end{document}